\documentclass[11pt]{amsart}
\usepackage{amsmath}
\usepackage{amsfonts,amssymb}
\usepackage{graphicx}
\usepackage{enumerate}
\usepackage{amsthm}
\usepackage[all]{xy}
\usepackage{color}
\usepackage{lmodern}
\usepackage{shuffle}
\usepackage{lmodern}
\usepackage{hyperref}
\usepackage{rotating} 
\usepackage{scalefnt}
\usepackage{axodraw}
\usepackage{xspace}
\usepackage{bm}
 
 \font\sevenrm=cmr7

\usepackage{a4wide}

\newtheorem{theorem}{Theorem}[section]
\newtheorem {lemma}[theorem]{Lemma}
\newtheorem {proposition}[theorem]{Proposition}
\newtheorem {definition}[theorem]{Definition}
\newtheorem {corollary}[theorem]{Corollary}

\newtheorem {remark}[theorem]{Remark}

\newcommand{\nc}{\newcommand}
\nc{\ignore}[1]{\relax}
\nc{\mrm}[1]{{\rm #1}}
\nc{\dirlim}{\displaystyle{\lim_{\longrightarrow}}\,}
\nc{\invlim}{\displaystyle{\lim_{\longleftarrow}}\,}
\nc{\vep}{\varepsilon} \nc{\ep}{\epsilon}
\nc{\sigmat}{\widetilde\sigma}
\nc{\ostar}{\overline{*}}

\nc{\mchar}{\mrm{Char}}
\nc{\Hom}{\mrm{Hom}}
\nc{\id}{\mrm{id}}

 \nc{\grad}[1]{^{({#1})}}
 \nc{\fil}[1]{_{#1}}

\nc{\BA}{{\Bbb A}} \nc{\CC}{{\Bbb C}} \nc{\DD}{{\Bbb D}}
\nc{\EE}{{\Bbb E}} \nc{\FF}{{\Bbb F}} \nc{\GG}{{\Bbb G}}
\nc{\HH}{{\Bbb H}} \nc{\LL}{{\Bbb L}} \nc{\NN}{{\Bbb N}}
\nc{\PP}{{\Bbb P}} \nc{\QQ}{{\Bbb Q}} \nc{\RR}{{\Bbb R}}
\nc{\TT}{{\Bbb T}} \nc{\VV}{{\Bbb V}} \nc{\ZZ}{{\Bbb Z}}
\nc{\Cal}[1]{{\mathcal {#1}}}
\nc{\mop}[1]{\mathop{\hbox {\rm #1} }}
\nc{\smop}[1]{\mathop{\hbox {\sevenrm #1} }}
\nc{\mopl}[1]{\mathop{\hbox {\rm #1} }\limits}
\nc{\frakg}{{\mathfrak g}}
\nc{\g}[1]{{\mathfrak {#1}}}
\nc{\wt}{\widetilde}
\nc{\wh}{\widehat}
\nc{\un}{\hbox{\bf 1}}
\nc{\redtext}[1]{\textcolor{red}{#1}}
\nc{\bluetext}[1]{\textcolor{blue}{#1}}

\nc\fleche[1]{\mathop{\hbox to #1 mm{\rightarrowfill}}\limits}
\def\semi{\mathrel{\times}\kern -.85pt\joinrel\mathrel{\raise
    1.4pt\hbox{${\scriptscriptstyle |}$}}}

\def\shu{\joinrel{\!\scriptstyle\amalg\hskip -3.1pt\amalg}\,}
\def\sshu{\joinrel{\,\scriptscriptstyle\amalg\hskip -2.5pt\amalg}\,}
\def\fleche#1{\mathop{\hbox to #1 mm{\rightarrowfill}}\limits}
\def\gfleche#1{\mathop{\hbox to #1 mm{\leftarrowfill}}\limits}
\def\inj#1{\mathop{\hbox to #1 mm{$\lhook\joinrel$\rightarrowfill}}\limits}
\def\ginj#1{\mathop{\hbox to #1 mm{\leftarrowfill$\joinrel\rhook$}}\limits}
\def\surj#1{\mathop{\hbox to #1 mm{\rightarrowfill\hskip 2pt\llap{$\rightarrow$}}}\limits}
\def\gsurj#1{\mathop{\hbox to #1 mm{\rlap{$\leftarrow$}\hskip 2pt
      \leftarrowfill}}\limits}
\def \restr#1{\mathstrut_{\textstyle |}\raise-6pt\hbox{$\scriptstyle #1$}}
\def \srestr#1{\mathstrut_{\scriptstyle |}\hbox to
-1.5pt{}\raise-4pt\hbox{$\scriptscriptstyle #1$}}

\newcommand{\lbutcher}{{\raise 3.9pt\hbox{$\circ$}}\hskip -1.9pt{\scriptstyle \searrow}}
\def\qshu{\!\joinrel{\!\scriptstyle\amalg\hskip -3.1pt\amalg}\,\hskip -8.2pt\hbox{-}\hskip 4pt}
\def\sqshu{\joinrel{\scriptscriptstyle\amalg\hskip -2.5pt\amalg}\,\hskip -7pt\hbox{-}\hskip 4pt}
\def\cto{\joinrel{\hbox{$\nearrow$}\hskip -3.65mm\raise 2pt\hbox{$\nearrow$}}}
\def\scto{\joinrel{\scalebox {0.8}{\hbox{$\nearrow$}}\hskip -2.54mm{\scalebox{0.8}{\raise 2pt\hbox{$\nearrow$}}}}}

\def\racine{{\scalebox{0.3}{
\begin{picture}(12,12)(38,-29)
\SetWidth{0.5} \SetColor{Black} \Vertex(45,-18){5.66}
\end{picture}}}}
  
 \def\arbrea{\,{\scalebox{0.15}{ 
  \begin{picture}(8,55) (370,-248)
    \SetWidth{2}
    \SetColor{Black}
    \Line(374,-244)(374,-200)
    \Vertex(374,-197){9}
    \Vertex(375,-245){12}
  \end{picture}
}}\,}

 \def\arbreba{\,{\scalebox{0.15}{ 
\begin{picture}(8,106) (370,-197)
    \SetWidth{2}
    \SetColor{Black}
    \Line(374,-193)(374,-149)
    \Vertex(374,-146){9}
    \Vertex(375,-194){12}
    \Line(374,-142)(374,-98)
    \Vertex(374,-95){9}
  \end{picture}
}}\,}

 \def\arbrebb{\,{\scalebox{0.15}{ 
  \begin{picture}(48,48) (349,-255)
    \SetWidth{2}
    \SetColor{Black}
    \Vertex(375,-252){12}
    \Line(376,-250)(395,-215)
    \Line(373,-251)(354,-214)
    \Vertex(353,-211){9}
    \Vertex(395,-213){9}
  \end{picture}
}}}

\def\arbreca{\,{\scalebox{0.15}{
\begin{picture}(8,156) (370,-147)
    \SetWidth{2}
    \SetColor{Black}
    \Line(374,-143)(374,-99)
    \Vertex(374,-96){9}
    \Vertex(375,-144){12}
    \Line(374,-92)(374,-48)
    \Vertex(374,-45){9}
    \Line(374,-42)(374,2)
    \Vertex(374,5){9}
  \end{picture}
}}\,}

\def\arbrecb{\,{\scalebox{0.15}{
\begin{picture}(48,94) (349,-255)
\SetWidth{2}
\SetColor{Black}
\Line(376,-204)(395,-169)
\Line(373,-205)(354,-168)
\Vertex(353,-165){9}
\Vertex(395,-167){9}
\Vertex(374,-205){9}
\Line(374,-246)(374,-209)
\Vertex(374,-252){12}
\end{picture}}}\,}

\def\arbrecc{\,{\scalebox{0.15}{
 \begin{picture}(48,98) (349,-205)
    \SetWidth{2}
    \SetColor{Black}
    \Vertex(375,-202){12}
    \Line(376,-200)(395,-165)
    \Line(373,-201)(354,-164)
    \Vertex(353,-161){9}
    \Vertex(395,-163){9}
    \Line(353,-160)(353,-113)
    \Vertex(353,-111){9}
  \end{picture}
}}\,}

\def\arbreccc{\,{\scalebox{0.15}{
 \begin{picture}(48,98) (349,-205)
    \SetWidth{2}
    \SetColor{Black}
    \Vertex(375,-202){12}
    \Line(376,-200)(395,-165)
    \Line(373,-201)(354,-164)
    \Vertex(353,-161){9}
    \Vertex(395,-163){9}
    \Line(395,-160)(395,-113)
    \Vertex(395,-111){9}
  \end{picture}
}}\,}

\def\arbrecd{\,{\scalebox{0.15}{
\begin{picture}(48,52) (349,-251)
    \SetWidth{2}
    \SetColor{Black}
    \Vertex(375,-248){12}
    \Line(376,-246)(395,-211)
    \Line(373,-247)(354,-210)
    \Vertex(353,-207){9}
    \Vertex(395,-209){9}
    \Line(375,-247)(375,-206)
    \Vertex(376,-203){9}
  \end{picture}
 }}\,}

 \def\arbreadec #1#2{\,{\scalebox{0.35}{ 
  \begin{picture}(8,55) (370,-248)
    \SetWidth{2}
    \SetColor{Black}
    \Line(374,-244)(374,-210)
    \Vertex(374,-210){5}
    \Vertex(374,-245){5}
     \Text(385,-240)[lb]{\Huge{\Black{$#1$}}}
    \Text(385,-205)[lb]{\Huge{\Black{$#2$}}}
  \end{picture}
}}\,}

\def\arbrebbLab{\,{\scalebox{0.32}{
  \begin{picture}(48,48) (349,-255)
    \SetWidth{2}
    \SetColor{Black}
    \Vertex(375,-252){8}
    \Line(376,-250)(395,-215)
    \Line(373,-251)(354,-214)
    \Vertex(353,-211){6}
    \Vertex(395,-211){6}
    \Text(370,-285)[lb]{\Huge{\Black{$u$}}}
    \Text(320,-215)[lb]{\Huge{\Black{$w$}}}
    \Text(415,-215)[lb]{\Huge{\Black{$v$}}}
  \end{picture}
}}}

\def\arbreebz{\,{\scalebox{0.25}{
\begin{picture}(90,120) (350,-275)
    \SetWidth{2}
    \SetColor{Black}
    \Vertex(400,-290){12}
    \Line(470,-130)(470,-210)
    \Vertex(470,-130){9}
    \Line(398,-299)(330,-210)
    \Vertex(330,-210){9}
    \Line(320,-205)(270,-130)
    \Vertex(270,-130){9}
    \Line(400,-299)(470,-210)
    \Vertex(470,-210){9}
    \Line(335,-210)(370,-130)
    \Vertex(370,-130){9}
    \Line(475,-210)(520,-130)
    \Vertex(520,-130){9}
    \Line(465,-210)(420,-130)
    \Vertex(420,-130){9}
    \Line(520,-130)(520,-40)
    \Vertex(520,-40){9}
    \Line(270,-130)(270,-40)
    \Vertex(270,-40){9}
    \Text(397,-325)[lb]{\Huge{\Black{$1$}}}
    \Text(465,-245)[lb]{\Huge{\Black{$2$}}}
    \Text(325,-245)[lb]{\Huge{\Black{$7$}}}
    \Text(250,-120)[lb]{\Huge{\Black{$9$}}}
    \Text(465,-115)[lb]{\Huge{\Black{$5$}}}
    \Text(365,-115)[lb]{\Huge{\Black{$8$}}}
    \Text(530,-120)[lb]{\Huge{\Black{$3$}}}
    \Text(415,-115)[lb]{\Huge{\Black{$6$}}}
    \Text(515,-25)[lb]{\Huge{\Black{$4$}}}
    \Text(260,-25)[lb]{\Huge{\Black{$10$}}}
  \end{picture}}}\,}

\def\bx{\mathbb X}
\def\diagramme #1{\vskip 4mm \centerline {#1} \vskip 4mm}

\begin{document}


\title[Rough differential equations on homogeneous Spaces]{Planarly branched Rough Paths and\\ rough differential equations on Homogeneous Spaces}


\author[C.~Curry]{C.~Curry}
\address{Dept.~of Mathematical Sciences, 
		Norwegian University of Science and Technology (NTNU), 
		7491 Trondheim, Norway.}
\email{charles.curry@ntnu.no}

\author[K.~Ebrahimi-Fard]{K.~Ebrahimi-Fard}
\address{Dept.~of Mathematical Sciences, 
		Norwegian University of Science and Technology (NTNU),
		7491 Trondheim, Norway.}
         \email{kurusch.ebrahimi-fard@ntnu.no}         
         \urladdr{https://folk.ntnu.no/kurusche/}

\author[D.~Manchon]{D.~Manchon}
\address{Univ. Blaise Pascal,
         	C.N.R.S.-UMR 6620,
         	63177 Aubi\`ere, France}       
         \email{manchon@math.univ-bpclermont.fr}
         \urladdr{http://math.univ-bpclermont.fr/~manchon/}

\author[H.~Z.~Munthe-Kaas]{H.~Z.~Munthe-Kaas}
\address{Dept.~of Mathematics,
         	University of Bergen,
         	Postbox 7800,
         	N-5020 Bergen, Norway}       
         \email{hans.munthe-kaas@uib.no}
         \urladdr{http://hans.munthe-kaas.no}
         
\begin{abstract}         
This work studies rough differential equations (RDEs) on homogeneous spaces. We provide an explicit expansion of the solution at each point of the real line using decorated planar forests. The notion of planarly branched rough path is developed, following Gubinelli's branched rough paths. The main difference being the replacement of the Butcher--Connes--Kreimer Hopf algebra of non-planar rooted trees by the Munthe-Kaas--Wright Hopf algebra of planar rooted forests. The latter underlies the extension of Butcher's $B$-series to Lie--Butcher series known in Lie group integration theory. Planarly branched rough paths admit the study of RDEs on homogeneous spaces, the same way Gubinelli's branched rough paths are used for RDEs on finite-dimensional vector spaces. An analogue of Lyons' extension theorem is proven. Under analyticity assumptions on the coefficients and when the H\"older index of the driving path is one, we show convergence of the planar forest expansion in a small time interval.         
\end{abstract}

%

\maketitle





\tableofcontents


\section{Introduction}
\label{sec:intro}

Given a set of vector fields $\{f_i\}_{i=1}^d$ on some $n$-dimensional smooth manifold $\Cal M$, we are interested in the controlled differential equation:
\begin{align}\label{eq:control1}
	dY_{st} &= \sum_{i=1}^d f_i(Y_{st}) dX_t^i,
\end{align}
with initial condition $Y_{ss} =y$, where the controls $t\mapsto X_t^i$ are differentiable, or even only H\"older-continuous real-valued functions\footnote{We have chosen any real $s$ as initial time rather than zero, whence the two-variable notation. Derivation is always understood with respect to the variable $t$, the first variable $s$ remaining inert. The two positive integers $n$ and $d$ are a priori unrelated.}. When $\Cal M$ is an affine space $\mathbb R^n$, rough path theory on $\mathbb R^d$, together with its branched version introduced by M.~Gubinelli \cite{Gubi2010}, is the correct setting to express the solutions of \eqref{eq:control1} when the controls are not differentiable. An important case of the latter situation is given by Brownian motion on $\mathbb R^d$, of which sample paths are almost surely nowhere differentiable\footnote{Brownian motion is almost surely of H\"older regularity $\gamma$ for any $\gamma <1/2$.}. The existence of a solution in a small interval around the point $s$ has been proven by Gubinelli, using the notion of controlled path in the branched setting \cite[Section 8]{Gubi2010} (see also  \cite[Section 3]{HaiKel2015}). The Taylor expansion of such a solution at any point is expressed by means of choosing a branched rough path $\mathbb X$ over the driving path $X=(X^1,\ldots, X^d)$. See, for example, the introduction of reference \cite{HaiKel2015} by M.~Hairer and D.~Kelly for a concise account.

\medskip

The theory of rough paths was introduced and developed by T.~Lyons' \cite{L98}. It is based on Chen's theory of iterated integrals \cite{Chen77} and provides an integration theory for solving differential equations driven by irregular signals. The intuitive idea of prescribing the path together with its iterated integrals is encapsulated by the definition of a rough path as a two-parameter family of Hopf algebra characters of the shuffle Hopf algebra $\Cal H_{{\shuffle}}^A$ over the finite alphabet $A=\{a_1,\ldots,a_d\}$, subject to precise estimates as well as to Chen's lemma. The latter is a lifting of the chain rule for integration. Gubinelli's branched rough paths are based on J.~Butcher's $B$-series from numerical integration theory, and are defined similarly to Lyons' rough paths, with the exception that the Hopf algebra at hand is the Butcher--Connes--Kreimer Hopf algebra $\Cal H_{\smop{BCK}}^A$ of $A$-decorated non-planar rooted trees.

\medskip

In a first step, this article introduces and develops the theory of rough paths on $\mathbb R^d$ for any connected graded Hopf algebra fulfilling rather mild assumptions with respect to its combinatorics. An analogue of Lyons' extension theorem is proven (Theorem \ref{extension-generalized}), using the Sewing Lemma as in the classical case (Proposition \ref{prop:fdpsewing}). In particular, following Gubinelli's approach we use the notion of Lie--Butcher series from Lie group integration theory to define \textsl{planarly branched rough paths} on $\mathbb R^d$ as rough paths in that generalised sense, for which the Hopf algebra at hand is the Hopf algebra of Lie group integrators $\Cal H_{\smop{MKW}}^A$ introduced in \cite{MunWri2008} by W.~Wright and one of the current authors. In a nutshell, this combinatorial Hopf algebra is 
linearly spanned by planar ordered rooted forests, possibly with decorations on the vertices. The product in this commutative Hopf algebra is the shuffle product of the forests, which are considered as words with planar rooted trees as letters. The coproduct is based on the notion of \textsl{left admissible cuts} on forests. We then argue that planarly branched rough paths provide the correct setting for understanding controlled differential equations on a homogeneous space, i.e., a manifold acted upon transitively by a finite-dimensional Lie group. To be more precise, it provides the means to write the Taylor expansion of a solution at each time, particularly suited to the underlying geometric setting.

\medskip

We conclude with a first discussion of the analytic aspects of differential equations driven by planarly branched rough paths. In this article, we restrict to considering the convergence of the full Taylor series on a small time interval (Corollary \ref{Taylor-cv}). This necessarily assumes analyticity of the vector fields, and makes use of Cauchy estimates in a similar manner to \cite[Section 5]{Gubi2010}. On the other hand, this method is limited to considering driving paths for which the H\"older index $\gamma$ of the control path $X$ is equal to one (Lipschitz case). A much more promising approach is to consider instead truncations of the Taylor expansion with controlled remainder, following Davie \cite{D08}, see also \cite{CW17} for the extension of this method to Lie series for the pullback flow. The main obstacle to this technique is the lack of results showing that iterative applications of approximate flows can be concatenated to give an approximation of controlled error on a larger compact time interval. This is equivalent to the existence of global error estimates for Lie group integrators. Such results are established in the forthcoming work \cite{CS2018}, the ramifications of which will be explored in a future sequel on the existence and uniqueness of solutions under much less restrictive assumptions.\\

The paper is organised as follows: in Section \ref{sect:fs} we write down the Taylor expansion of the solution of \eqref{eq:control1} on a homogeneous space in the case of differentiable controls, using a Picard iteration. \textcolor{blue}{We then introduce} a suitable class of combinatorial Hopf algebras in Section \ref{sect:comb-fact}, defining a notion of factorial adapted to this general setting. Then we define in Section \ref{sect:H-rough} a functorial notion of $\gamma$-regular rough path associated to any combinatorial Hopf algebra in the above sense, and we prove Lyons' extension theorem in this setting, along the lines of reference \cite{FP2006}. After giving a brief account on Lie--Butcher theory in Section \ref{sec:flows}, we recall the Munthe-Kaas--Wright Hopf algebra $\Cal H_{\smop{MKW}}^A$ of Lie group integrators. We recall in Section \ref{sec:MKWHA} two relevant combinatorial notions associated with planar forests, namely three partial orders on the set of vertices \cite{A14}, and the planar forest factorial $\sigma\mapsto \sigma!$ of \cite{MF17}, which matches the general notion of factorial mentioned above. Planarly branched rough paths are then defined as rough paths associated to the particular Hopf algebra $\Cal H_{\smop{MKW}}^A$. Section \ref{sec:quotient} is devoted to a canonical surjective Hopf algebra morphism $\mathfrak a_\ll$ from $\Cal H_{\smop{MKW}}^A$ onto the shuffle Hopf algebra $\Cal H^A_{\sshu}$, the \textsl{planar arborification}, in the sense of J.~Ecalle's notion of arborification. A contracting version of planar arborification is also given, where the shuffle Hopf algebra is replaced by a quasi-shuffle Hopf algebra. Finally, Section \ref{sec:RDEHom} deals with rough differential equations on homogeneous spaces driven by a H\"older-continuous path $X$. Any planarly branched rough path above $X$ yields a corresponding formal solution. Following the lines of thought of \cite[Section 5]{Gubi2010} (see also \cite[Proposition 1.8]{BS2016}), we prove convergence of the planar forest expansion in a small interval at each time, under an appropriate analyticity assumption on the coefficients $f_i$, when the driving path is Lipschitz, i.e., of H\"older regularity $\gamma=1$. \textcolor{blue}{An account of the sewing lemma is given in the Appendix.}\\

\noindent\textbf{Acknowledgements:} We would like to thank Lorenzo Zambotti and Ilya Chevyrev for crucial discussions and comments which led to substantial improvements of this paper, in particular by pointing us to the recent article \cite{B17}. We also thank \textcolor{blue}{Igor Mencattini,} Alexander Schmeding and Rosa Preiss for helpful comments. The third author greatly acknowledges the warm hospitality and stimulating working conditions which he experienced at NTNU in Trondheim and at Bergen University during his visits in May 2017. He also would like to thank Fr\'ed\'eric Fauvet for illuminating discussions on J.~Ecalle's notion of arborification. \textcolor{blue}{Finally, we thank the referee for very pertinent suggestions and remarks.} The article received support from Campus France, PHC Aurora 40946NM.


\section{Formal series expansion of the solution}
\label{sect:fs}

The theory of numerical integration algorithms on Lie groups and manifolds \cite{Iserles00} has been developed over the last two decades. In this context new algebraic structures were revealed which combine Butcher's $B$-series \cite{HWL_02} and Lie-series into Lie--Butcher series on manifolds \cite{LMK_13}. Brouder's work \cite{Brouder_00} showed that Hopf and pre-Lie algebras of non-planar rooted trees provide the algebraic foundation of $B$-series. For Lie--Butcher series the new concepts of post-Lie algebras and the Munthe-Kaas--Wright Hopf algebra are the foundations. These are examples of algebraic combinatorial structures which arise naturally from the geometry of connections on homogenous spaces.

\smallskip

We rewrite the differential equation \eqref{eq:control1} in the following form:
\begin{equation}
\label{rde}
	dY_{st}=\sum_{i=1}^d \# f_i(Y_{st})\,dX_t^i
\end{equation}
with initial condition $Y_{ss}=y$, where the unknown is a path $Y_s \colon t \mapsto Y_{st}$ in a homogeneous space $\Cal M$, with transitive action $(g,x)\mapsto g.x$ of a Lie group $G$ on it. The control path $X \colon t \mapsto X_t = (X_t^1,\ldots,X_t^d)$ with values in $\mathbb R^d$ is given, and the $f_i$'s are smooth maps from $\Cal M$ into the Lie algebra $\frak g=\mop{Lie}(G)$, which in turn define smooth vector fields $x \mapsto \# f_i(x)$ on $\Cal M$:
\begin{equation*}
	\# f_i(x):=\frac{d}{dt}{\restr{t=0}}\exp \big(tf_i(x)\big).x\in T_x \Cal M.
\end{equation*}
In the language of Lie algebroids, considering the tangent vector bundle and the trivial vector bundle $E=\Cal M\times\mathfrak g$, the map $\# \colon C^\infty(\Cal M,\mathfrak g)\to\chi(\Cal M)$ is the composition on the left with the anchor map $\rho \colon E\to T\Cal M$ defined by $\rho (x,X):=\frac{d}{dt}{\restr{t=0}}(\exp tX).x$.

\medskip

The central point of our approach is based on formally lifting the differential equation \eqref{rde} to the space $C^\infty\big(\Cal M,\mathcal U(\frak g)\big)[[h]]$, where $\mathcal U(\mathfrak g)$ is the universal enveloping algebra of $\mathfrak g$. This is achieved as follows: setting $t=s+h$, we denote by $\varphi_{st}$ the formal diffeomorphism defined by $\varphi_{st}(Y_{ss}):=Y_{st}$,  where $t\mapsto Y_{st}$ is the solution of the initial value problem \eqref{rde}. This formal diffeomorphism can be expressed as:
\begin{equation*}
	\varphi_{st}=\# \bm Y_{st}
\end{equation*}
with $\bm Y_{st} \in C^\infty\big(\Cal M,\mathcal U(\frak g)\big)[[h]]$. It turns out that there exists a non-commutative associative product $*$ on $C^\infty\big(\Cal M,\mathcal U(\frak g)\big)$, distinct from the pointwise product in $\Cal U(\mathfrak g)$, which reflects the composition product of differential operators on $\Cal M$, in the sense that:
\begin{equation*}
	\#(u*v)=\# u\circ\# v
\end{equation*}
for any $u,v\in C^\infty\big(\Cal M,\mathcal U(\frak g)\big)$. See reference \cite{MunWri2008} for details. The unit is the constant function $\bm 1$ on $\Cal M$ equal to $1\in\Cal U(\mathfrak g)$, and $\#\bm 1$ is the identity operator. The existence of this product is a direct consequence of the post-Lie algebra structure on $C^\infty\big(\Cal M,\frak g)$. The reader may consult \cite{KAH2015} for details. Extending this product to formal series, our lifting of \eqref{rde} is written as:
\begin{equation}
\label{rde-lifted}
	d\bm Y_{st} = \sum_{i=1}^d \bm Y_{st}*f_i\,dX_t^i
\end{equation}
with initial condition $\bm Y_{ss}=\bm 1$. The non-commutative product $*$ is the extension of the Grossman--Larson product on the post-associative algebra $C^\infty \big(\Cal M,\mathcal U(\frak g)\big)$ to formal series, which reflects the composition of differential operators \cite{MunWri2008}. A full account of the post-Lie algebra structure on $C^\infty \big(\Cal M,\frak g)$ and the post-associative algebra structure on $C^\infty \big(\Cal M,\mathcal U(\frak g)\big)$ will be provided further below in Section \ref{sec:flows}. Let us just mention at this stage that for any $f,g\in C^\infty \big(\Cal M,\frak g)$ we have (Leibniz' rule):
\begin{equation}
\label{GL-1}
	f*g=fg+f\rhd g,
\end{equation}
where $fg$ stands for the pointwise product in $C^\infty \big(\Cal M,\mathcal U(\frak g)\big)$, and where $f\rhd g$ stands for $\#(f).g$. The solution of \eqref{rde-lifted} is a formal diffeomorphism, i.e., it verifies $\bm Y_{st}\rhd(\rho\psi)=(\bm Y_{st}\rhd \rho)(\bm Y_{st}\rhd\psi)$ for any $\rho,\psi\in C^\infty(\Cal M)$. The formal path $Y_{st}$ solving the initial value problem \eqref{rde}, with initial condition $Y_{ss}=y$, is then the character of $C^\infty(\Cal M)$ with values in $\mathbb R[[h]]$ given for any $\psi \in C^\infty(\Cal M)$ by:
\begin{eqnarray}
\label{evaluation}
	Y_{st} \colon C^\infty(\Cal M)&\longrightarrow &\mathbb R[[h]]\nonumber\\
		   \psi &\longmapsto      &\psi(Y_{st})=(\bm Y_{st}\rhd\psi)(y).
\end{eqnarray}
Plugging \eqref{rde-lifted} into \eqref{evaluation} yields:
\allowdisplaybreaks
\begin{eqnarray*}
	\frac{d}{dt}\psi(Y_{st})
	&=&\frac{d}{dt}(\bm Y_{st}\rhd\psi)(y)\\
	&=&\big(\big(\bm Y_{st}*F(t)\big)\rhd\psi\big)(y)\\
	&=&\big(\bm Y_{st}\rhd\big(F(t)\rhd\psi\big)\big)(y)\\
	&=&\big(F(t)\rhd\psi\big)(Y_{st}),
\end{eqnarray*}
which proves this assertion, and therefore justifies viewing \eqref{rde-lifted} as a lift of \eqref{rde}. We refer to $\psi(Y_{st})$ as the evaluation of $\psi$ on the formal path $Y_{st}$. Equation \eqref{rde-lifted} can be written in integral form:
\begin{eqnarray}
\label{rde-gl-int}
	\bm Y_{st}  &=&\bm 1+\int_s^t \bm Y_{su}*F(u)\,du\nonumber\\
			  &=&\bm 1+\sum_{i=1}^d\int_s^t \bm Y_{su}*f_i\,dX^i_u.
\end{eqnarray}
A simple Picard iteration gives the formal expansion:
\begin{align}
	\bm Y_{st}
	&=\bm 1+\sum_{n\ge 1}\,\sum_{1\le i_1,\ldots,i_n\le d}\left(\int\cdots
	\int_{s\le t_n\le\cdots\le t_1\le t}f_{i_n}*\cdots *f_{i_1}\,dX^{i_1}_{t_1}\cdots dX^{i_n}_{t_n}\right) \nonumber\\ 
	&=\bm 1+\sum_{n\ge 1}\,\sum_{1\le i_1,\ldots,i_n\le d}\left(\int\cdots
	\int_{s\le t_n\le\cdots\le t_1\le t}dX^{i_1}_{t_1}\cdots dX^{i_n}_{t_n}\right)f_{i_n}* \cdots * f_{i_1}. \label{formalsol}
\end{align}
Using word notation, where $f_w$ stands for the monomial $f_{i_n}*\cdots *f_{i_1}$ when the word $w$ is given by $a_{i_1}\cdots a_{i_n}$, the formal expansion \eqref{formalsol} will be written as a word series
\begin{equation}
\label{expansion}
	\bm Y_{st}=\sum_{w\in A^*}\langle\bx_{st},w\rangle f_w.
\end{equation}
Using \eqref{GL-1}, the first terms of the expansion are:
\allowdisplaybreaks
\begin{align*}
	\lefteqn{\bm Y_{st}=\bm 1+\sum_{i=1}^d \langle\bx_{st},a_i\rangle f_i
	+\sum_{i,j=1}^d\langle\bx_{st},a_ia_j\rangle (f_jf_i+f_j\rhd f_i)}\\
	&+\sum_{i,j,k=1}^d\langle\bx_{st},a_ia_ja_k\rangle 
	\Big(f_kf_jf_i
	+(f_k\rhd f_j)f_i
	+f_k(f_j\rhd f_i)
	+f_j(f_k\rhd f_i)
	+(f_kf_j)\rhd f_i
	+(f_k\rhd f_j)\rhd f_i\Big)\\
	&+O(h^4).
\end{align*}
We observe that the number of components in the term of order three on the right-hand side can be reduced from six to five:
\begin{align}
\label{six-to-five}
	\sum_{i,j,k=1}^d
	&\bigg[\langle\bx_{st},a_ia_ja_k\rangle \Big(f_kf_jf_i
				+(f_k\rhd f_j)f_i+(f_kf_j)\rhd f_i+(f_k\rhd f_j)\rhd f_i\Big)\nonumber\\
	&+\langle\bx_{st},a_ia_ja_k+a_ia_ka_j\rangle f_j(f_k\rhd f_i)\bigg],
\end{align}
which corresponds to the five planar rooted decorated forests with three vertices, displayed in the following order:
\begin{equation}
\label{five-forests}
	\racine_k\racine_j\racine_i\hskip 	8mm 
	\arbrea_j^k\racine_i\hskip 			8mm 
	\arbrebb_i^{\,j\hskip -7mm k}\hskip 	8mm
	{\arbreba_i^j}^{\hskip -4.5pt k}\hskip 	8mm 
	\racine_j\arbrea_i^k.
\end{equation}
The appearance of planar rooted forests relates to a natural further step in abstraction, namely using the \textsl{Lie--Butcher series formalism}. It consists in an additional lifting of equation \eqref{rde-lifted} to the free post-associative algebra, i.e., the universal enveloping algebra over the free post-Lie algebra in $d$ generators, more precisely to its completion $(\Cal H_{\smop{MKW}}^A)^*$. We obtain then the so-called fundamental differential equation:
\begin{equation}
\label{rde-postlifted}
	d\mathbb Y_{st} = \sum_{i=1}^d \mathbb Y_{st}*\racine_i\,dX_t^i
\end{equation}
with initial condition $\mathbb Y_{ss}=\bm 1$, where $*$ is now the non-commutative convolution (Grossman--Larson) product of two linear forms on $\Cal H_{\smop{MKW}}^A$. Suppose for the moment that the path $X$ in $\mathbb{R}^d$ is differentiable. Equation \eqref{rde-postlifted} can then be re-written as:
\begin{equation}
\label{rde-gl}
	\dot{\mathbb Y}_{st}=\frac{d}{dt}{\mathbb Y}_{st}=\sum_{i=1}^d\dot X_t^i\mathbb Y_{st}*\racine_i,
\end{equation}
with initial condition $\mathbb Y_{ss}=\bm 1$. For any $s,t$ the so-called fundamental  solution $\mathbb Y_{st}$ of \eqref{rde-postlifted} is given by
\begin{equation}
\label{expansion-word}
	\mathbb Y_{st}=\sum_{\ell\ge 0}\sum_{w=a_1\cdots a_\ell\in A^*}\langle\bx_{st},w\rangle \racine_{a_\ell}*\cdots*\racine_{a_1}.
\end{equation}

\medskip

The coefficient of the last component in \eqref{six-to-five} is obtained by integrating $dX^{i_1}_{t_1}dX^{i_2}_{t_2}dX^{i_3}_{t_3}$ on the union of two simplices $\{(t_1,t_2,t_3),\ s\le t_3\le t_1,t_2\le t\}$. This domain is associated to the decorated forest $\racine_j\arbrea_i^k$ by means of a partial order $\ll$ on the vertices described in Subsection \ref{ssec:partialorders}, which is closely related to the notion of left-admissible cuts for the coproduct in $\Cal H_{\smop{MKW}}^A$. The order $\ll$ is total on the four other planar forests of degree three appearing in \eqref{five-forests}, hence the corresponding coefficients are obtained by integrating on a single simplex. Integrating over these domains lifts $\bx_{st}$ to a two-parameter family of characters of the Hopf algebra $\Cal H_{\smop{MKW}}^A$, which still verifies Chen's lemma. This calls for considering rough differential equations defined on the homogeneous space $\Cal M$ driven by planarly branched rough paths.

\medskip

Further below we will use J.~Ecalle's notion of arborification to write the Taylor expansion of the solution \eqref{expansion}, or rather its abstract counterpart \eqref{expansion-word} in its planar arborified form:
\begin{equation}
\label{expansion-lie-butcher}
	\mathbb Y_{st}=\sum_{\sigma\in F^A_{\smop{pl}}}\langle\wt\bx_{st},\sigma\rangle \sigma
\end{equation}
with $\wt\bx_{st}:=\bx_{st}\circ \mathfrak a_\ll$ where $(\mathbb X_{st})_{s,t\in\mathbb R}$ is the signature of the path $X$, and where $ F^A_{\smop{pl}}$ stands for the set of $A$-decorated planar rooted forests.


\section{Factorials in combinatorial Hopf algebras}
\label{sect:comb-fact}

We consider the notion of factorial in the context of a fairly general class of combinatorial algebras. This concept will encompass the usual factorial of positive integers, the tree and forest factorials as well as a planar version of the latter.


\subsection{Inverse-factorial characters in connected graded Hopf algebras}
\label{sect:inv-fact}

Let $\mathcal H = \bigoplus_{n \ge 0} \mathcal H_n$ be any connected graded Hopf algebra over some field $\mathbf k$ of characteristic zero, and let $\alpha \colon \mathcal H_1\to \mathbf{k}$ be a nonzero linear map. The degree of an element $x \in \mathcal H$ is denoted $|x|$. The \textsl{inverse-factorial character} $q_\alpha$ associated to these data is defined by
\begin{itemize}
	\item $q_\alpha(\mathbf{1})=1$,
	\item $q_\alpha\restr{\mathcal H_1}=\alpha$,
	\item $q_\alpha*q_\alpha(x)=2^{|x|}q_\alpha(x)$ for any $x\in\mathcal H$.
\end{itemize}
It is indeed given for any homogeneous $x$ of degree $|x| \ge 2$ by the recursive procedure:
\begin{equation}\label{rec-inv-fact}
	q_\alpha(x)=\frac{1}{2^{|x|}-2}\sideset{}{'}\sum_{(x)}q_\alpha(x')q_\alpha(x'').
\end{equation}
See \cite[Section 7]{Gubi2010} in the particular case of the Butcher--Connes--Kreimer Hopf algebra. Here $\Delta'(x):=\sideset{}{'}\sum_{\!\!(x)}x' \otimes x''$ denotes the reduced coproduct of $\mathcal H$ (in Sweedler's notation), and the full coproduct is $\Delta(x)=\Delta'(x)+x\otimes \mathbf{1} + \mathbf{1} \otimes x=\sum_{(x)}x_{1}\otimes x_{2}$. The multiplicativity property $q_\alpha(xy)=q_\alpha(x)q_\alpha(y)$ is verified recursively with respect to $\ell=|x|+|y|\ge 2$ (the cases $\ell=0$ and $\ell=1$ are immediately checked):
\begin{eqnarray*}
	q_\alpha(xy)
		&=&\frac{1}{2^{|x|+|y|}}\sum_{(x),(y)}q_\alpha(x_{1}y_{1})q_\alpha(x_{2}y_{2})\\
		&=&\frac{1}{2^{|x|+|y|}}\left(\sum_{(x),(y)}q_\alpha(x_{1})q_\alpha(y_{1})q_\alpha(x_{2})q_\alpha(y_{2})
				-2q_\alpha(x)q_\alpha(y)+2q_\alpha(xy)\right)\\
		&=&q_\alpha(x)q_\alpha(y)-\frac{1}{2^{|x|+|y|-1}}\Big(q_\alpha(x)q_\alpha(y)-q_\alpha(xy)\Big),
\end{eqnarray*}
hence $q_\alpha(x)q_\alpha(y)-q_\alpha(xy)=0$.

In general the linear form $\alpha$ is fixed once and for all, and $q_\alpha$ will be abbreviated to $q$. We remark that in concrete situations there is a natural linear basis for the degree one component $\mathcal H_1$ (and more generally for $\mathcal H$, see Paragraph \ref{sect:cha} below for a precise setting), and $\alpha$ will be the linear form on $\mathcal H_1$ which takes the value $1$ on each element of the basis. Taking as a simple example the shuffle algebra on an alphabet $A$, the binomial formula:
\begin{equation*}
	\frac{2^n}{n!}=\sum_{p=0}^n\frac{1}{p!}\frac{1}{(n-p)!}
\end{equation*}
shows that $q_\alpha(w)=1/|w|!$ where $\alpha(a)=1$ for each letter $a\in A$. This example justifies the terminology chosen.

\begin{proposition}
For any $h\in\mathbf k$ the $h$-th power convolution of $q=q_\alpha$ makes sense as a character of $\mathcal H$, and admits the following explicit expression:
\begin{equation}\label{q-to-h}
	q^{*h}(x)=h^{|x|}q(x).
\end{equation}
\end{proposition}

\begin{proof}
One can express $q$ as $\varepsilon+\kappa$ with $\kappa(\mathbf 1)=0$. Then we have for any $x\in\mathcal H$
\begin{equation*}
	q^{*h}(x)=(\varepsilon+\kappa)^{*h}(x)=\sum_{p\ge 0}{h\choose p}\kappa^{*p}(x).
\end{equation*}
The right-hand side is a finite sum, owing to the co-nilpotence of the coproduct. The expression $q^{*h}(xy)-q^{*h}(x)q^{*h}(y)$ is polynomial in $h$ and vanishes at any non-negative integer $h$, hence vanishes identically. Similarly, the expression $q^{*h}(x)-h^{|x|}q(x)$ is polynomial in $h$ and vanishes at any $h=2^N$ where $N$ is a non-negative integer, hence vanishes identically.
\end{proof}

The following corollary generalises both the binomial formula and Gubinelli's branched binomial formula \cite[Lemma 4.4]{Gubi2010}.

\begin{corollary}\label{hopf-binom}
For any $h,k\in \mathbf k$ and for any homogeneous element $x\in\mathcal H$, the following Hopf-algebraic binomial formula holds:
\begin{equation*}
	q(x)(h+k)^{|x|}=\sum_{(x)}q(x_1)q(x_2)h^{|x_1|}k^{|x_2|}.
\end{equation*}
\end{corollary}

\begin{proof}
It is a straightforward application of \eqref{q-to-h} together with the group property $q^{*h}*q^{*k}=q^{*(h+k)}$.
\end{proof}

Inverse-factorial characters are functorial, that is, if $\Cal H$ and $\Cal H'$ are two connected graded Hopf algebras and if $\Phi:\Cal H\to\Cal H'$ is a morphism of Hopf algebras preserving the degree, then for any linear map $\alpha':\Cal H'\to \mathbf{k}$ we have:
\begin{equation}
\label{funct}
	q_{\alpha}=q_{\alpha'}\circ \Phi
\end{equation}
where $\alpha:=\alpha'\circ \Phi\restr{\mathcal H_1}$.


\subsection{A suitable category of combinatorial Hopf algebras}
\label{sect:cha}

Although the theory of combinatorial Hopf algebras constitutes an active field of research, with duly acknowledged applications in discrete mathematics, analysis, probability, control, and quantum field theory, no general consensus has yet emerged on a proper definition of those Hopf algebras. Saying this, we propose here a definition which will match our purpose, i.e., give estimates which will ensure convergence of the formal solutions of our singular differential equations in some particular cases. A different proposal for a definition of combinatorial Hopf algebras can also be found in \cite{DS2016} (see Definition 3.1 therein). In both definitions, a privileged linear basis is part of the initial data.\\

\begin{definition}
A combinatorial Hopf algebra is a graded connected Hopf algebra $\mathcal H=\bigoplus_{n\ge 0}\mathcal H_n$ over a field $\mathbf{k}$ of characteristic zero, together with a basis $\mathcal B=\bigsqcup_{n\ge 0}\mathcal B_n$ of homogeneous elements, such that
\begin{enumerate}
	\item There exist two positive constants $B$ and $C$ such that the dimension of $\mathcal H_n$ is bounded by $BC^n$ (in other words, the Poincar\'e--Hilbert series of $\mathcal H$ converges in a small disk around the origin).
	
	\item The structure constants $c_{\sigma\tau}^\rho$ and $c_{\rho}^{\sigma\tau}$ of the product and the coproduct, defined for any $\sigma,\tau,\rho \in \mathcal B$ respectively by
$$
	\sigma\tau=\sum_{\rho\,\in\mathcal B}c_{\sigma\tau}^\rho \rho,\hskip 12mm \Delta \rho
	=\sum_{\sigma,\tau\,\in\mathcal B}c_{\rho}^{\sigma\tau}\sigma\otimes\tau$$
are non-negative integers (which vanish unless $|\sigma|+|\tau|=|\rho|$)\ignore{, and are bounded by $B'C'^{|\rho|}$ where $B'$ and $C'$ are two positive constants}.
\end{enumerate}
\end{definition}

In any combinatorial Hopf algebra in the above sense, the inverse-factorial character $q$ will be chosen such that $q(\tau)=1$ for any $\tau\in \mathcal B$ of degree one. We adopt the natural shorthand notation:
\begin{equation}\label{gen-factorial}
	\tau!=\frac{1}{q(\tau)}
\end{equation}
for any $\tau\in \mathcal B$. The two main examples are the shuffle Hopf algebra $\mathcal H^A_{\sshu}$ and the Butcher--Connes--Kreimer Hopf algebra $\mathcal H^A_{\smop{BCK}}$ over a finite alphabet $A$. In the former case the basis $\mathcal B$ is given by words with letters in $A$, whereas in the latter case we have non-planar rooted forests decorated by $A$. On $\mathcal H^A_{\sshu}$ the corresponding factorial is the usual factorial of the length of a word. On $\mathcal H^A_{\smop{BCK}}$ it is the usual forest factorial \cite[Lemma 4.4]{Gubi2010}. A third major example is the Hopf algebra $\mathcal H^A_{\smop{MKW}}$ of Lie group integrators described in Paragraph \ref{sect:hmkw} below.

\medskip

Let $(\mathcal H,\mathcal B)$ and $(\mathcal H',\mathcal B')$ be two combinatorial hopf algebras in the above sense. A Hopf algebra morphism $\Phi:(\mathcal H,\mathcal B)\to (\mathcal H',\mathcal B')$ is \textsl{combinatorial} if it is of degree zero and if, for any $\tau\in\mathcal B$, the element $\Phi(\tau)\in \mathcal H'$ is a linear combination of elements of the basis $\mathcal B'$ with non-negative integer coefficients. Combinatorial Hopf algebras in the above sense together with combinatorial morphisms form a category. The forgetful functor  $(\mathcal H,\mathcal B)\mapsto \mathcal H$ into the category of connected graded Hopf algebras is given by forgetting the basis.

\begin{remark}\label{nondeg}\rm
The inverse-factorial character $q$ may vanish on some elements $\tau$ of the basis, yielding $\tau!=+\infty$. This happens if and only if $\tau$ is primitive of degree $n\ge 2$. We therefore call a combinatorial Hopf algebra \textsl{non-degenerate} if
\begin{equation}
	\mathcal B\cap\mop{Prim}(\mathcal H)=\mathcal B_1.
\end{equation}
The three combinatorial Hopf algebras $\mathcal H^A_{\sshu}$, $\mathcal H^A_{\smop{BCK}}$ and $\mathcal H^A_{\smop{MKW}}$ happen to be non-degenerate. Examples of degenerate combinatorial Hopf algebras can easily be found among Hopf algebras of Feynman graphs, as primitive multiloop Feynman graphs do exist.
\end{remark}


\section{Rough paths and connected graded Hopf algebras}
\label{sect:H-rough}

We show that Lyons' definition of rough paths \cite{L98} extends straightforwardly when replacing the shuffle Hopf algebra with any commutative connected graded Hopf algebra. In particular a naturally extended version of the extension theorem \cite[Theorem 2.2.1]{L98} is available.

\subsection{Chen iterated integrals and rough paths}
\noindent Let $d$ be a positive integer, and let us consider a smooth path in $\mathbb R^d$
\allowdisplaybreaks
\begin{eqnarray*}
	X:\mathbb R 	&\longrightarrow
				& \mathbb R^d\\
			      t &\longmapsto & X(t)=\big(X_1(t),\ldots,X_d(t)\big).
\end{eqnarray*}
Let $\Cal H^A$ be the algebra of the free monoid $A^*$ generated by the alphabet $A:=\{a_1,\ldots,a_d\}$, and augmented with the empty word $\bm 1$ as unit. Let $\bx_{st}\in(\Cal H^A)^\star$ be defined for any $s,t\in\mathbb R$ and word $w=a_{j_1}\cdots a_{j_n} \in A^*$ by $n$-fold iterated integrals:
\begin{eqnarray}\label{signature}
	\langle \bx_{st},\,a_{j_1}\cdots a_{j_n}\rangle 
	&:=&\int\cdots\int_{s\le t_n\le\cdots \le t_1\le t}\,\dot X_{j_1}(t_1)\cdots \dot X_{j_n}(t_n)\,dt_1\cdots dt_n\nonumber\\
	&=&\int\cdots\int_{s\le t_n\le\cdots \le t_1\le t}\,dX_{j_1}(t_1)\cdots dX_{j_n}(t_n).
\end{eqnarray}
This is extended to the empty word $\bm 1$ by $\langle \bx_{st},\bm 1\rangle:=1$. Suppose moreover that the derivative $\dot X$ is bounded, i.e., $\mop{sup}_{j=1}^d\mop{sup}_{t\in\mathbb R}|\dot X_j(t)|=C<+\infty$. The volume of the simplex
$$
	\Delta^n_{[s,t]}:=\{(t_1,\ldots, t_n),\  s \le t_n \le \cdots \le t_1 \le t\}
$$
over which the iterated integration \eqref{signature} of length $n$ is performed is equal to $|t-s|^n/n!$, which yields the following estimate for any word $w\in A^*$:
\begin{equation}
\label{sp-estimates}
	\mop{sup}_{s\neq t}\frac{\vert\langle \bx_{st},\,w\rangle\vert}{\vert t-s\vert^{\vert w\vert}}\le\frac{C^{|w|}}{|w|!},
\end{equation}
where $\vert w\vert$ stands for the length of the word $w$, i.e., its number of letters. It turns out \cite{Chen54, Chen57, Chen67, Chen77, Ree58} that $\bx_{st}$ is a two-parameter family of characters with respect to the shuffle product of words, namely:
\begin{equation}\label{shuffle}
	\langle \bx_{st},\,v\rangle\langle \bx_{st},\,w\rangle=\langle \bx_{st},\,v\shu w\rangle.
\end{equation}
The shuffle product $\shu$ is defined inductively by $w \shu \bm 1 =\bm 1\shu w=w$ and
\begin{equation}
\label{shuffleproduct}
	(a_i v) \shu (a_j w)=a_i(v\shu a_j w) + a_j(a_i v \shu w),
\end{equation}
for all words $v,w \in A^\ast$ and letters $a_i,a_j\in A$. The resulting shuffle algebra is denoted $\Cal H_{\sshu}^A$. For instance, $a_i \shu a_j = a_ia_j + a_ja_i$ and
\begin{equation*}
	a_{i_1} a_{i_2}\shu a_{i_3}a_{i_4} = a_{i_1}a_{i_2}a_{i_3}a_{i_4}
	+ a_{i_1}a_{i_3}a_{i_2}a_{i_4}+a_{i_1}a_{i_3}a_{i_4}a_{i_2}
	+ a_{i_3}a_{i_1}a_{i_2}a_{i_4} + a_{i_3}a_{i_1}a_{i_4}a_{i_2} + a_{i_3}a_{i_4}a_{i_1}a_{i_2}.
\end{equation*} 
Moreover, the following property, now widely referred to as ``Chen's lemma", is verified:
\begin{equation}\label{chen-lemma}
	\langle\bx_{st},\,v\rangle=\sum_{v'v''=v}\langle\bx_{su},\,v'\rangle\langle\bx_{ut},\,v''\rangle.
\end{equation}
The sum on the righthand side extends over all splittings of the word $v \in A^*$ into two words, $v'$ and $v''$, such that the concatenation $v'v''$ equals $v$. Both properties are easily shown by a suitable decomposition of the integration domain into smaller pieces with Lebesgue-negligible mutual intersections: a product of two simplices is written as a union of simplices for proving \eqref{shuffle}, and a simplex of size $t-s$ is written as a union of products of simplices of respective size $u-s$ and $t-u$ for proving \eqref{chen-lemma} when $s\le u\le t$. The latter is advantageously re-written in terms of the convolution product associated to the deconcatenation coproduct $\Delta: \Cal H_{\sshu}^A \to \Cal H_{\sshu}^A \otimes \Cal H_{\sshu}^A$ defined on words, $w\mapsto \Delta(v)=\sum_{v'v''=v}v'\otimes v''$. The latter turns the shuffle algebra $\Cal H_{\sshu}^A$ into a connected graded commutative Hopf algebra $(\Cal H_{\sshu}^A,\shu,\Delta)$ with convolution product defined on the dual $({\Cal H_{\sshu}^A})^\star$:
\begin{equation}\label{chen-lemma-hopf}
	\bx_{st}=\bx_{su}\ast\bx_{ut}=m_{\mathbb{R}}(\bx_{su}\otimes\bx_{ut})\Delta.
\end{equation}
It has long ago been proposed by K.~T.~Chen to call ``generalized path" \cite{Chen67} any two-parameter family of characters of the shuffle algebra $\Cal H_{\sshu}^A$ verifying \eqref{shuffle} and \eqref{chen-lemma} together with a mild continuity assumption. Lyons introduced the seminal notion of rough path \cite{L98}, which can be defined as follows \cite[Definition 1.2]{HaiKel2015}: a \textsl{geometric rough path\footnote{To be precise: \textsl{weak geometric rough path}, see \cite[Remark 1.3]{HaiKel2015}.} of regularity $\gamma$}, with $0<\gamma\le 1$, is a generalized path in the sense of Chen, satisfying moreover the estimates:
\begin{equation}\label{rp-estimates}
	\mop{sup}_{s\neq t}\frac{\vert\langle \bx_{st},\,w\rangle\vert}{\vert t-s\vert^{\gamma\vert w\vert}}<C(w)
\end{equation}
for any word $w$ of length $\vert w\vert$, where $C(w)$ is some positive constant. The evaluation on length one words is then given by the increments of a $\gamma$-H\"older continuous path:
\begin{equation}
	\langle\bx_{st},\,a_j\rangle:=X_j(t)-X_j(s)
\end{equation}
with $X_j(t):=\langle\bx_{t_0t},\,a_j\rangle$ for some arbitrary choice of $t_0\in\mathbb R$. Iterated integrals \eqref{signature} cannot be given any sense for any $n>0$ if the path is only of regularity $\gamma\le 1/2$. Lyons' extension theorem \cite{L98,HaiKel2015}, however, stipulates that the collection of coefficients $\langle \bx_{st},\,w\rangle$ for the words $w$ of length up to $[1/\gamma]$ completely determines the $\gamma$-regular rough path $\bx$. This result is a particular case of Theorem \ref{extension-generalized} below, the proof of which also uses the sewing lemma (Proposition \ref{prop:fdpsewing}).


\subsection{Rough paths generalized to commutative combinatorial Hopf algebras}
\label{sect:rpha}

We have briefly indicated in the Introduction how to adapt the notion of rough path to any connected graded Hopf algebra. Here is the precise definition:

\begin{definition}\label{rough}
Let $\Cal H=\bigoplus_{n\ge 0}\Cal H_n$ be a commutative graded Hopf algebra with unit $\bm 1$, connected in the sense that $\Cal H_0$ is one-dimensional, and let $\gamma\in]0,1]$. We suppose that $\mathcal H$ is endowed with a homogeneous basis $\mathcal B$ making it combinatorial and non-degenerate in the sense of Section \ref{sect:cha}. A {\rm $\gamma$-regular $\Cal H$-rough path\/} is a two-parameter family $\bx=(\bx_{st})_{s,t\in\mathbb R}$ of linear forms on $\Cal H$ such that $\langle \bx_{st},\bm 1\rangle=1$ and
\begin{enumerate}[I)]
\item \label{lyons-one} for any $s,t\in\mathbb R$ and for any $\sigma,\tau$ in $\Cal H$, the following equality holds 
$$
	\langle\bx_{st},\,\sigma\tau\rangle=\langle\bx_{st},\,\sigma \rangle\langle\bx_{st},\, \tau\rangle,
$$

\item \label{lyons-two} for any $s,t,u\in\mathbb R$, Chen's lemma holds
$$
	\bx_{su}*\bx_{ut}=\bx_{st},
$$ 
where the convolution $\ast$ is the usual one defined in terms to the coproduct on $\Cal H$, 

\item \label{lyons-three} for any $n\ge 0$ and for any $\sigma \in \Cal B_n$, we have the estimates
\begin{equation}\label{plrp-estimates}
	\mop{sup}_{s\neq t}\frac{\vert\langle \bx_{st},\,\sigma\rangle\vert}{\vert t-s\vert^{\gamma |\sigma|}}<C(\sigma).
\end{equation}
\end{enumerate}
\end{definition}

The notion of $\gamma$-regular branched rough path \cite{Gubi2010, HaiKel2015} is recovered by choosing for $\Cal H$ the Butcher--Connes--Kreimer Hopf algebra $\Cal H_{\smop{BCK}}^A$ of $A$-decorated (non-planar) rooted forests. Recall that the product in this Hopf algebra is given by the disjoint union of rooted trees.
\textcolor{blue}
{\begin{remark}\rm
Theorem \ref{extension-generalized} below will permit to give precise expressions of the constants $C(w)$ and $C(\sigma)$ in Estimates \eqref{rp-estimates} and \eqref{plrp-estimates}, respectively.
\end{remark}}
\smallskip

\noindent The truncated counterpart of $\Cal H$-rough paths is defined as follows.

\begin{definition}\label{rough-truncated}
Let $N$ be a positive integer and let $\Cal H^{(N)}:=\bigoplus_{k=0}^N \Cal H_k$. Let $\gamma\in]0,1]$. A {\rm $\gamma$-regular $N$-truncated $\Cal H$-rough path\/} is a two-parameter family $\bx=(\bx_{st})_{s,t\in\mathbb R}$ of linear forms on $\Cal H^{(N)}$ such that:
\begin{enumerate}[i)]
\item \label{item-one-truncated} the multiplicativity property \eqref{lyons-one} above holds for any $\sigma\in \Cal H_{p}$ and $\tau \in \Cal H_{q}$ with $p+q\le N$,

\item \label{item-two-truncated} Chen's lemma \eqref {lyons-two} holds, where the convolution refers to the restriction of the coproduct to $\Cal H^{(N)}$,

\item \label{item-three-truncated} the estimates \eqref{lyons-three} hold for any $\sigma \in \Cal H_n$ with $n\le N$.
\end{enumerate}
\end{definition}

For later use we also recall Sweedler's notation $\Delta(\sigma)=\sum_{(\sigma)}\sigma_{1}\otimes \sigma_{2},$ for the full coproduct $\Delta$ in $\Cal H$, as well as its iterated versions $\Delta^{(k-1)}(\sigma)=\sum_{(\sigma)}\sigma_1\otimes\cdots\otimes \sigma_k.$ For the reduced coproduct, we also adopted a Sweedler-type notation:
$$
	\Delta'(\sigma) := \Delta(\sigma) - \sigma \otimes\bm 1 - \bm 1 \otimes \sigma = \sideset{}{'}\sum_{(\sigma)} \sigma' \otimes \sigma''.
$$

Lyons' extension theorem \cite[Theorem 2.2.1]{L98} can be generalised to this setting, with basically the same proof:

\begin{theorem}\label{extension-generalized}
Let $\gamma\in]0,1]$, and let $N:=[1/\gamma]$. Any $\gamma$-regular $N$-truncated $\Cal H$-rough path admits a unique extension to a $\gamma$-regular $\Cal H$-rough path. Moreover, there exists a positive constant $c$ such that the following estimate holds:
\begin{equation}\label{estimate-gamma-rough}
\vert\langle\mathbb X_{st},\sigma\rangle\vert\le c^{|\sigma|}q_\gamma(\sigma)|t-s|^{\gamma|\sigma|}
\end{equation}
for any $\sigma\in\mathcal B$, with $q_\gamma(\sigma)=q(\sigma)$ for $|\sigma|\le N$ and
\begin{equation}\label{recursive-q-gamma}
q_\gamma(\sigma):=\frac{1}{2^{\gamma|\sigma|}-2}\sideset{}{'}\sum_{(\sigma)}q_\gamma(\sigma')q_\gamma(\sigma'')
\end{equation}
for $|\sigma|\ge N+1$.
\end{theorem}

\begin{proof}
Notice that $\varepsilon:=\gamma(N+1)-1$ is (strictly) positive. If the element $\sigma\in\Cal H$ is homogeneous of degree $n$, then $\sigma'$ and $\sigma''$ in its reduced coproduct, $\Delta'(\sigma)$, can be taken homogeneous with respective degree $p$ and $q$ with $p+q=n$ and $p,q\le n-1$. Now let $(\bx_{st})_{s,t\in\mathbb R}$ be a $\gamma$-regular $N$-truncated $\Cal H$-rough path. We extend it trivially to $\Cal H^{(N+1)}$ by setting $\langle\bx_{st},\,\sigma\rangle=0$ for any $\sigma\in\Cal H_{N+1}$.\\

\noindent Now let $\sigma\in \Cal B_{N+1}$. Fix an arbitrary $o\in\mathbb R$, and consider the function of two real variables:
\begin{equation}
	\mu(s,t):=\sideset{}{'}\sum_{(\sigma)}\langle\bx_{os},\sigma'\rangle\langle\bx_{st},\sigma''\rangle.
\end{equation}
Let $s,t,u\in\mathbb R$. A simple computation yields:
\begin{align}
\lefteqn{\mu(s,u)+\mu(u,t)-\mu(s,t) 
	=\sideset{}{'}\sum_{(\sigma)}-\langle\bx_{os},\sigma'\rangle\langle\bx_{st},\sigma''\rangle+\langle\bx_{os},\sigma'\rangle\langle\bx_{su},\sigma''\rangle
	+ \langle\bx_{ou},\sigma'\rangle\langle\bx_{ut},\sigma''\rangle}											\nonumber\\
	&=-\sideset{}{'}\sum_{(\sigma)}\langle\bx_{os},\sigma'\rangle\langle\bx_{su},\sigma''\rangle\langle\bx_{ut},\sigma'''\rangle
	-\sideset{}{'}\sum_{(\sigma)}\langle\bx_{os},\sigma'\rangle\langle\bx_{su},\sigma''\rangle
	-\sideset{}{'}\sum_{(\sigma)}\langle\bx_{os},\sigma'\rangle\langle\bx_{ut},\sigma''\rangle \nonumber\\
	&\hskip 10mm +\sideset{}{'}\sum_{(\sigma)}\langle\bx_{os},\sigma'\rangle\langle\bx_{su},\sigma''\rangle\nonumber\\
	&\hskip 10mm +\sideset{}{'}\sum_{(\sigma)}\langle\bx_{os},\sigma'\rangle\langle\bx_{su},\sigma''\rangle\langle\bx_{ut},\sigma'''\rangle
	+\sideset{}{'}\sum_{(\sigma)}\langle\bx_{os},\sigma'\rangle\langle\bx_{ut},\sigma''\rangle
	+\sideset{}{'}\sum_{(\sigma)}\langle\bx_{su},\sigma'\rangle\langle\bx_{ut},\sigma''\rangle	\nonumber\\
	&=\sideset{}{'}\sum_{(\sigma)}\langle\bx_{su},\sigma'\rangle\langle\bx_{ut},\sigma''\rangle.\label{mu}
\end{align}
Hence if $s\le u\le t$ or $t\le u\le s$, we can estimate, using \eqref{recursive-q-gamma}:
\begin{eqnarray*}
	\vert\mu(s,t)-\mu(s,u)-\mu(u,t)\vert
	&\le& \sideset{}{'}\sum_{(\sigma)}\vert\langle\bx_{su},\sigma'\rangle\langle\bx_{ut},\sigma''\rangle\vert\\
	&\le& \sideset{}{'}\sum_{(\sigma)}q_\gamma(\sigma')c^{|\sigma'|}|u-s|^{\gamma|\sigma'|}q_\gamma(\sigma'')c^{|\sigma''|}|t-u|^{\gamma|\sigma''|}.\\
	\ignore{
	&\le& c^{|\sigma|}(2^{\gamma|\sigma|}-2)q_\gamma(\sigma)\mop{sup}_{0\le a\le \gamma|\sigma|}\hskip 2mm\mop{sup}_{u\in[\smop{min}(s,t),\,\smop{max}(s,t)]}|t-u|^a|u-s|^{\gamma|\sigma|-a}\\
	&\le & c^{|\sigma|}(2^{\gamma|\sigma|}-2)q_\gamma(\sigma)\mop{sup}_{0\le a\le \gamma|\sigma|} \left(\frac{a}{\gamma|\sigma|}\right)^a\left(\frac{\gamma|\sigma|-a}{\gamma|\sigma|}\right)^{\gamma|\sigma|-a}|t-s|^{\gamma|\sigma|}\\
	&\le&c^{|\sigma|}(2^{\gamma|\sigma|}-2)\textcolor{blue}{2^{-\gamma\sigma}\hskip -7mm //\hskip 4mm}q_\gamma(\sigma)|t-s|^{\gamma|\sigma|}.}
\end{eqnarray*}
Here we have chosen the constant $c$ such that the estimates \eqref{estimate-gamma-rough} hold for any $\sigma\in\mathcal B_n,\, n\le N$. This is possible since the sets $\mathcal B_n$ are finite. It follows from $\gamma|\sigma|=\gamma(N+1)=1+\varepsilon$ and the Sewing Lemma (Proposition \ref{prop:fdpsewing}) that there exists a unique map $\varphi$ defined on $\mathbb R$, up to an additive constant, such that:
\begin{align}
\label{estimate-phi}
	\vert\varphi(t)-\varphi(s)-\mu(s,t)\vert
	&\le \frac{c^{|\sigma|}}{2^{\gamma|\sigma|}-2}\sum_{(\sigma)}q_\gamma(\sigma')q_\gamma(\sigma'')|t-s|^{\gamma|\sigma|}\\
	&\phantom{==} =c^{|\sigma|}q_\gamma(\sigma)|t-s|^{\gamma|\sigma|}. 
\end{align}
Now defining:
\begin{eqnarray}
	\langle\wt\bx_{st},\sigma\rangle &:=& 
	\begin{cases} 
	\langle\bx_{st},\sigma\rangle \hbox{,\ for }\sigma\in\Cal H_n,\, n\le N,\\
	\varphi(s)-\varphi(t)-\mu(s,t) \hbox{,\ for }\sigma\in\Cal H_{N+1},\label{def-xtilde}
	\end{cases}
\end{eqnarray}
we immediately get from \eqref{mu}:
\begin{equation}
\label{chen-N-plus-one}
	\langle\wt\bx_{st}-\wt\bx_{su}-\wt\bx_{ut},\,\sigma\rangle 
	=\sideset{}{'}\sum_{(\sigma)}\langle\bx_{su},\sigma'\rangle\langle\bx_{ut},\sigma''\rangle.
\end{equation}
From \eqref{estimate-phi} and \eqref{def-xtilde} we then have
\begin{equation}
\label{estimate-wtxst}
	\vert\langle\wt\bx_{st},\sigma\rangle\vert\le c^{|\sigma|}q_\gamma(\sigma)|t-s|^{\gamma|\sigma|}.
\end{equation}
Let us now check item \eqref{item-one-truncated} in Definition \ref{rough-truncated} for $\wt\bx_{st}$ for any $\sigma\in\Cal H_p$ and $\tau\in\Cal H_q$ with $p+q=N+1$. We split the interval $[s,t]$ (or $[t,s]$) into $k$ sub-intervals $[s_j,s_{j+1}]$ of equal length $|t-s|/k$, with $s_0:=\mop{inf}(s,t)$ and $s_k:=\mop{sup}(s,t)$. From Chen's lemma up to degree $N+1$ stemming from \eqref{chen-N-plus-one}, we can compute, supposing $s\le t$ here:
\begin{align*}
	\lefteqn{\langle\wt\bx_{st},\sigma\tau\rangle-\langle\wt\bx_{st},\sigma\rangle\langle\wt\bx_{st},\tau\rangle}\\
	&=\langle\wt\bx_{s_0s_1}\ast\cdots\ast\wt\bx_{s_{k-1}s_k},\,\sigma\tau\rangle
	-\langle\wt\bx_{s_0s_1}\ast\cdots\ast\wt\bx_{s_{k-1}s_k},
	\,\sigma\rangle\langle\wt\bx_{s_0s_1}\ast\cdots\ast\wt\bx_{s_{k-1}s_k},\,\tau\rangle\\
	&=\sum_{(\sigma),(\tau)}\bigg(\prod_{j=0}^{k-1}\langle \wt\bx_{s_js_{j+1}},\sigma_j\tau_j\rangle-\prod_{j=0}^{k-1}\langle
	 \wt\bx_{s_js_{j+1}},\sigma_j\rangle\langle \wt\bx_{s_js_{j+1}},\tau_j\rangle\bigg).
\end{align*}
The term under the summation sign vanishes unless there is a $j \in \{0,\ldots,k-1\}$ such that $\sigma_j=\sigma$ and $\tau_j=\tau$, in which case we have $\sigma_i=\tau_i=\bm 1$ for $i\neq j$. Hence,
\begin{equation*}
	\langle\wt\bx_{st},\sigma\tau\rangle-\langle\wt\bx_{st},\sigma\rangle\langle\wt\bx_{st},\tau\rangle
	=\sum_{j=0}^{k-1}
	\langle\wt\bx_{s_js_{j+1}},\sigma\tau\rangle-\langle\wt\bx_{s_js_{j+1}},\sigma\rangle\langle\wt\bx_{s_js_{j+1}},\tau\rangle.
\end{equation*}
From \eqref{estimate-wtxst} and item \eqref{item-three-truncated} of Definition \ref{rough-truncated} we get
\begin{equation*}
	\vert\langle\wt\bx_{st},\sigma\tau\rangle-\langle\wt\bx_{st},\sigma\rangle\langle\wt\bx_{st},\tau\rangle\vert
	\le kc^{|\sigma|}q_\gamma(\sigma)\left \vert \frac{t-s}{k}\right\vert^{1+\varepsilon}
	= k^{-\varepsilon} c^{|\sigma|}q_\gamma(\sigma)\vert t-s\vert^{1+\varepsilon}.
\end{equation*}
Hence $\langle\wt\bx_{st},\sigma\tau\rangle=\langle\wt\bx_{st},\sigma\rangle\langle\wt\bx_{st},\tau\rangle$ by letting $k$ go to $+\infty$. The same argument works \textsl{mutatis mutandis} in the case $s>t$. Hence, estimate \eqref{estimate-gamma-rough} is proven for $\langle \wt\bx,\sigma\rangle$ for any $\sigma\in\mathcal B_n,\, n\le N+1$.\\

Uniqueness can be proven by a similar argument. Indeed, suppose that $\overline \bx$ is another $(N+1)$-truncated $\gamma$-regular $\Cal H$-rough path extending $\bx$, and let $\delta_{st}:=\wt\bx_{st}-\overline\bx_{st}$ for any $s,t\in\mathbb R$. For any $\sigma\in\Cal H_{N+1}$ we have then the following:
\allowdisplaybreaks
\begin{eqnarray*}
	\langle \delta_{st},\sigma\rangle
	&=&\langle \wt\bx_{ss_1}\ast\cdots\ast \wt\bx_{s_{k-1}t}
	-\overline\bx_{ss_1}\ast\cdots\ast \overline\bx_{s_{k-1}t},\,\sigma\rangle\\
	&=&\langle \delta_{ss_1}+\cdots+\delta_{s_{k-1}t},\,\sigma\rangle.
\end{eqnarray*}
As we have $\vert\langle \delta_{s_js_{j+1}},\sigma\rangle\vert\le \overline C \vert \frac{t-s}{k}\vert^{1+\varepsilon}$ for some constant $\overline C$, we get $\vert\langle \delta_{st},\sigma\rangle\vert\le \overline C \vert t-s\vert^{1+\varepsilon}k^{-\varepsilon}$, hence $\langle \delta_{st},\sigma\rangle=0$ by letting $k$ go to infinity.

Iterating this process at any order finally yields a fully-fledged $\gamma$-regular $\Cal H$-rough path $\wt\bx$ extending $\bx$. 
\end{proof}

\begin{remark}\rm
Considering the striking similarity of the map $q_\gamma$ with the inverse-factorial character, which is nothing but $q_\gamma$ for $\gamma=1$, Gubinelli conjectured \cite[Remark 7.4]{Gubi2010} the following comparison, in the special case of the Butcher--Connes--Kreimer Hopf algebra (corresponding to branched rough paths):
\begin{equation}\label{estimate-gubi}
	q_\gamma(\sigma)\le \frac{BC^{|\sigma|}}{(\sigma!)^\gamma}
\end{equation}
for any $\sigma$ in $\mathcal B$ (i.e., any decorated forest in this particular case of $\Cal H_{\smop{BCK}}^A$), where $B$ and $C$ are positive constants. This conjecture has been recently proven by H.~Boedihardjo \cite{B17}. In the case of the shuffle Hopf algebra (corresponding to geometric rough paths), it happens to be a consequence of Lyons' neoclassical inequality (\cite[Theorem 2.1.1]{L98}, see also \cite[Remark 7.4]{Gubi2010}). It would be interesting to prove a similar result for a general class of combinatorial Hopf algebra, in particular for the Hopf algebra of Lie group integrators $\Cal H_{\smop{MKW}}^A$ defined in paragraph \ref{sect:hmkw} below, corresponding to the notion of rough paths we call \textsl{planarly branched rough paths}, defined in Paragraph \ref{sect:pbrp-def}. Our definition of $q_\gamma$ differs form that of Gubinelli's in the initial conditions $q_\gamma(\sigma)=q(\sigma)$ versus $q_\gamma^{\smop{Gub}}(\sigma)=1$ for any $\sigma\in \mathcal B_n$, $\,n\le N$. This choice is dictated by the functorial considerations of Paragraph \ref{sect:rpcha} below. In practice, one has very often $q(\sigma)\le 1$ for any element $\sigma\in \mathcal B_n,$ $n\le N$, which yields $q_\gamma(\sigma)\le q_\gamma^{\smop{Gub}}(\sigma)$ for any $\sigma\in\mathcal B$ by induction. Hence the majorations for $q_\gamma^{\smop{Gub}}$ obtained in \cite{B17} in the branched case also hold for our $q_\gamma$.
\end{remark}


\subsection{Factorial decay estimates}
\label{ssect:factdecay}

The linear map $q_\gamma$ defined in the statement of Theorem \ref{extension-generalized} is uniquely defined by $q_\gamma(\sigma)=1$ for any $\sigma\in\mathcal B_1\cup\{\mathbf 1\}$ and the recursive equations
\begin{eqnarray*}
	q_\gamma(\sigma):= 
	\begin{cases}
	\frac{1}{2^{|\sigma|}}q_\gamma* q_\gamma(\sigma) \hbox{,\ for }2\le|\sigma|\le N,\\
	\frac{1}{2^{\gamma|\sigma|}}q_\gamma* q_\gamma(\sigma)\hbox{,\ for }|\sigma|\ge N+1.
	\end{cases}
\end{eqnarray*}
As a consequence, $q_\gamma$ has the same functorial properties than the inverse-factorial character $q$, namely if $(\Cal H,\Cal B)$ and $(\Cal H', \Cal B')$ are two connected graded Hopf algebras and if $\Phi:\Cal H\to\Cal H'$ is a morphism of Hopf algebras preserving the degree, then we have for any $\gamma\in]0,1]$, with self-explanatory notations:
\begin{equation}
\label{funct-qgamma}
	q_{\gamma}=q'_{\gamma}\circ \Phi.
\end{equation}

\begin{proposition}\cite[Remark 7.4]{Gubi2010}
Let $\Cal H^A_{\sshu}$ be the shuffle Hopf algebra on a finite alphabet $A$. Then the following estimates hold: for any word $w\in A^*$,
\begin{equation}
\label{estimate-qgamma-shuffle}
	|q_\gamma(w)|\le \frac{C_\gamma^{|w|-1}}{(|w|!)^\gamma}
\end{equation}
where $C_\gamma$ is a positive real number depending only on $\gamma$.
\end{proposition}

\begin{proof}
Recall Gubinelli's variant of Lyons' neo-classical inequality: there exists $c_\gamma>0$ such that
\begin{equation}
\label{neoclassical}
	\sum_{k=0}^n\frac{a^{\gamma k}b^{\gamma(n-k)}}{(k!)^\gamma[(n-k)!]^\gamma}
	\le c_\gamma\frac{(a+b)^{n\gamma}}{(n!)^\gamma}.
\end{equation}
Now set 
$$
	C_\gamma:=\mop{sup}\left(1,\,\frac{2^{\gamma(N+1)}}{2^{\gamma(N+1)}-2}c_\gamma\right),
$$
and proceed by induction on the length of the word $w$. The case $|w|\le N$ being obvious, suppose $|w|\ge N+1$. We can compute, using \eqref{neoclassical} in the particular case $a=b=1$:
\begin{eqnarray*}
	q_\gamma(w)	
	&=&\frac{1}{2^{\gamma|w|}-2}\sideset{}{'}\sum_{(w)}q_\gamma(w')q_\gamma(w'')\\
	&\le&\frac{1}{2^{\gamma|w|}-2}\sideset{}{'}\sum_{(w)}\frac{C_\gamma^{|w'|-1}}{(|w'|!)^\gamma}\frac{C_\gamma^{|w''|-1}}{(|w''|!)^\gamma}\\
	&\le&\frac{1}{2^{\gamma|w|}-2}\sum_{(w)}\frac{C_\gamma^{|w_1|-1}}{(|w_1|!)^\gamma}\frac{C_\gamma^{|w_2|-1}}{(|w_2|!)^\gamma}\\
	&\le& \frac{C_\gamma^{|w|-2}}{2^{\gamma|w|}-2}\, c_\gamma \frac{2^{|w|\gamma}}{(|w|!)^\gamma}\\
	&\le & \frac{C_\gamma^{|w|-1}}{(|w|!)^\gamma}.
\end{eqnarray*}
\end{proof}

\begin{corollary}\label{estimate-qgamma-comb}
Let $(\Cal H,\Cal B)$ be a combinatorial Hopf algebra endowed with a combinatorial morphism $\Phi:(\Cal H,\Cal B)\to (\Cal H^A_{\sshu}, \Cal B')$, where $A$ is a finite alphabet and $\Cal B'=A^*$ is the standard basis of words. Then for any $\sigma\in\Cal B$ the following estimate holds:
\begin{equation}\label{estimate-qgamma-general}
	|q_\gamma(\sigma)|\le C_\gamma^{|\sigma|-1}\frac{(|\sigma|!)^{1-\gamma}}{\sigma!}
\end{equation}
with the same $C_\gamma$ as above.
\end{corollary}

\begin{proof}
For any $\sigma\in\Cal B$ we have $\Phi(\sigma)=\sum_{w\in A^*}b_w^\sigma w$, and we have by functoriality of the inverse factorial:
\begin{equation*}
	\sum_{w\in A^*}b_w^\sigma=\frac{|\sigma|!}{\sigma!}.
\end{equation*}
The proof relies on a simple computation using functoriality of $q_\gamma$ as well as the non-negativity of the coefficients $b_w^\sigma$, together with the fact that the $L^p$ norms are nondecreasing with respect to  $p\in]0,1]$ for probability measures:
\allowdisplaybreaks
\begin{eqnarray*}
	|q_\gamma(\sigma)|^{1/\gamma}
	&=&\left(\sum_{w\in A^*}b_w^\sigma q_\gamma(w)\right)^{1/\gamma}\\
	&=&\left(\frac{|\sigma|!}{\sigma!}\right)^{1/\gamma}\left(\frac{\sigma!}{|\sigma|!}\sum_{w\in A^*}b_w^\sigma q_\gamma(w)\right)^{1/\gamma}\\
	&\le & \left(\frac{|\sigma|!}{\sigma!}\right)^{1/\gamma-1}\sum_{w\in A^*}b_w^\sigma q_\gamma(w)^{1/\gamma}\\
	&\le& \left(\frac{|\sigma|!}{\sigma!}\right)^{1/\gamma-1}\sum_{w\in A^*}C_\gamma^{(|w|-1)/\gamma}\frac{b_w^\sigma}{|w|!}\\
	&\le&C_\gamma^{(|\sigma|-1)/\gamma}\left(\frac{|\sigma|!}{\sigma!}\right)^{1/\gamma-1}\frac{1}{\sigma!}\\
	&\le&\left(C_\gamma^{|\sigma|-1}\frac{(|\sigma|!)^{1-\gamma}}{\sigma!}\right)^{1/\gamma}.
\end{eqnarray*}
\end{proof}
We remark that H.~Boedihardjo recently obtained a much better estimate in the context of Gubinelli's branched rough paths, i.e., for the Butcher--Connes--Kreimer Hopf algebra, see \cite[Theorem 4]{B17}.


\subsection{Rough paths and combinatorial Hopf algebras}
\label{sect:rpcha}

We shall examine further properties of rough paths in the generalised sense given in Paragraph \ref{sect:rpha}, i.e., when the Hopf algebra at hand is combinatorial.

\begin{proposition}\label{funct-rough}
\begin{enumerate}
	\item Let $(\mathcal H,\mathcal B)$ be a combinatorial Hopf algebra in the sense of Paragraph \ref{sect:cha} and let $q$ be the associated inverse factorial character. Then $q(x)=\frac{1}{x!}$ is a (possibly vanishing) non-negative rational number for any $x\in\mathcal B$.

	\item Let $(\mathcal H,\mathcal B)$ and $(\mathcal H',\mathcal B')$ be two combinatorial Hopf algebras, and let $\Phi:(\mathcal H,\mathcal B)\to(\mathcal H',\mathcal B')$ be a combinatorial Hopf algebra morphism. Then the pull-back $\mathbb X_{st}:=\mathbb X'_{st}\circ\Phi$ of any $\gamma$-regular $\mathcal H'$-rough path $\mathbb X'_{st}$ is a $\gamma$-regular $\mathcal H$-rough path.
\end{enumerate}
\end{proposition}

\begin{proof}
Recall that $q(x)=1$ for any $x\in \mathcal B_1$. The first assertion is then recursively derived from equation \eqref{rec-inv-fact}. Multiplicativity as well as Chen's Lemma are immediate consequences of the fact that $\Phi$ is a Hopf algebra morphism. We now check the estimate for any $x\in\mathcal B_n$ with $n\ge 0$:
\begin{eqnarray*}
	\vert\langle\mathbb X_{st},\,x\rangle \vert
	&=&\vert\langle\mathbb X'_{st},\,\Phi(x)\rangle\vert\\
	&\le&\sum_{y\in\mathcal B'_n}b^x_y\vert\langle\mathbb X'_{st},\,y\rangle\vert\\
	&\le&C^n\sum_{y\in\mathcal B'_n}q_\gamma(y)b^x_y\vert t-s\vert^{\gamma n}.
\end{eqnarray*}
The proof of functoriality of the inverse factorial character \eqref{funct} can be easily adapted to its counterpart $q_\gamma$, although it is generally not a character when $\gamma$ differs from $1$. Hence we can derive the desired estimate:
\begin{equation}
	\vert\langle\mathbb X_{st},\,x\rangle \vert\le {C^n}q_\gamma(x)\vert t-s\vert^{\gamma n}.
\end{equation}

Note that we have used the non-negativity of the coefficients $b^x_y$ of the matrix of $\Phi$ expressed in the bases $\mathcal B$ and $\mathcal B'$.
\end{proof}



\section{Lie--Butcher theory}
\label{sec:flows}

Butcher's \emph{B-series} are a special form of Taylor expansion indexed by trees. They have become a fundamental tool for analysing numerical integration algorithms. The numerical analysis of general Lie group methods requires the generalisation of the $B$-series theory to so-called Lie--Butcher series, which are based on planar rooted forest, possibly decorated.


\subsection{Rooted trees and forests}
\label{ssect:trees}

For any positive integer $n$, a rooted tree of degree $n$ is a finite oriented tree with $n$ vertices. One of them, called the root, is a distinguished vertex without any outgoing edge. Any vertex can have arbitrarily many incoming edges, and any vertex other than the root has exactly one outgoing edge. Vertices with no incoming edges are called leaves. A planar rooted tree is a rooted tree together with an embedding in the plane. A planar rooted forest is a finite ordered collection of planar rooted trees. Here are the planar rooted forests up to four vertices:
$$\emptyset\hskip 7mm
	\racine\hskip 7mm
	\arbrea\hskip 3mm 
	\racine\racine\hskip 7mm
	\arbreba\hskip 3mm
	\arbrebb\hskip 3mm
	\arbrea\racine\hskip 3mm
	\racine\arbrea\hskip 3mm
	\racine\racine\racine\hskip 7mm
	\arbreca\hskip 3mm
	\arbrecb\hskip 3mm
	\arbrecc\hskip 3mm
	\arbreccc\hskip 3mm
	\arbrecd\hskip 3mm
	\arbreba\racine\hskip 3mm
	\racine\arbreba\hskip 3mm
	\arbrebb\racine\hskip 3mm
	\racine\arbrebb\hskip 3mm
	\arbrea\arbrea\hskip 3mm
	\arbrea\racine\racine\hskip 3mm
	\racine\arbrea\racine\hskip 3mm
	\racine\racine\arbrea\hskip 3mm 
	\racine\racine\racine\racine
$$
Let $A$ be any set. An $A$-decorated planar rooted forest is a pair $\sigma=(\overline\sigma,\varphi)$ where $\overline\sigma$ is a planar forest, and where $\varphi$ is a map from the vertex set $V(\overline\sigma)$ into $A$. We denote by $T_A^{\smop{pl}}$ (respectively $F_A^{\smop{pl}}$) the set of all $A$-decorated planar rooted trees (respectively forests), and by $\Cal{T}_A^{\smop{pl}}$ (respectively $\Cal{F}_A^{\smop{pl}}$) the linear space spanned by the elements of $T_A^{\smop{pl}}$ (respectively $F_A^{\smop{pl}}$).\\

$A$-decorated non-planar rooted forests are denoted by $\wt\sigma=(\overline{\wt\sigma},\varphi)$, where $\varphi$ is the decoration and $\overline{\wt\sigma}$ is the underlying non-planar forest. When $A$ is reduced to one element the notion of $A$-decoration is superfluous. Hence any $A$-decorated forest can be identified with its overlined counterpart.\\

\noindent Every $A$-decorated planar rooted tree can be written as:
\begin{equation}
\label{B+}
	\sigma = B_a^+(\sigma_1\, \cdots \,\sigma_k),
\end{equation}
where $B_a^+$ is the operation on forests which grafts each connected component $\sigma_i$ of a planar rooted forest $\sigma_1 \cdots\sigma_k$ on a common root decorated by $a\in A$. Note that in numerical analysis the bracket notation for $\sigma = [\sigma_1\, \cdots \,\sigma_k]_a$ is often used instead of the $B_a^+$ operator.


\subsection{Post-Lie and post-associative algebras}
\label{ssect:postLie}

A \textsl{left post-Lie algebra} \cite{V2007, KAH2015} is a vector space $A$ (over some field $\mathbf{k}$) together with two bilinear maps $[-,-]$ and $\rhd$ from $A\otimes A$ to $A$ such that
\begin{itemize}
	\item $[-,-]$ is a Lie bracket, i.e., it is antisymmetric and verifies the Jacobi identity.
	\item For any $a,b,c\in A$ we have
$$
	a\rhd [b,c]=[a\rhd b,c]+[a,b\rhd c].
$$

\item For any $a,b,c\in A$ we have
$$
	[a,b]\rhd c=a\rhd(b\rhd c)-(a\rhd b)\rhd c-b\rhd(a\rhd c)+(b\rhd a)\rhd c.
$$
\end{itemize}
\textcolor{blue}{The bracket $[\![-,-]\!]$ defined by $[\![a,b]\!]:=[a,b]+a\rhd b-b\rhd a$ is another Lie bracket on $A$.} The particular case when the Lie bracket $[-,-]$ vanishes on $A$ is referred to as \textsl{left pre-Lie algebra}. See \cite{Manchon2011} for details. Associative counterparts of post-Lie algebras are referred to as \textsl{post-associative algebras}. They first appear under the terminology "D-algebras" in \cite{MunWri2008}. A post-associative algebra is a vector space $B$ endowed with two linear maps $\cdot$ and $\rhd$ from $B\otimes B$ to $B$, a filtration $B^0=\mathbf{k.1}\subset B^1\subset B^2\subset\cdots$ with $B=\bigcup_j B^j$, and an augmentation $\varepsilon: B\to\hskip -3mm\to \mathbf{k}$ such that
\begin{enumerate}

	\item $L_{\mathbf{1}}=\mop{Id}_B$, and $a\rhd \mathbf{1}=0$ for any $a\in\mop{Ker}\varepsilon$.
	
	\item The product $\cdot$ is associative with unit $\mathbf{1}$, and $B^p\cdot B^q\subset B^{p+q}$ for any $p,q\ge 0$.

	\item $A:=B^1\cap\mop{Ker}\varepsilon$ is stable \textcolor{blue}{under the product $\rhd$ as well as} under the Lie bracket obtained by anti-symmetrisation of the associative product, and generates the unital associative algebra $(B,\cdot$).\label{d-alg-gen}

	\item\label{d-alg-der} For any $a,b,c\in B$ with $a\in A$ we have
$$
	a\rhd (b \cdot c)=(a\rhd b)\cdot c+b\cdot (a\rhd c).
$$

\item\label{d-alg-induction} For any $a,b,c\in B$ with $a\in A$ we have
$$
	(a \cdot b)\rhd c=a\rhd(b\rhd c)-(a\rhd b)\rhd c.
$$
\end{enumerate}
In particular, $A$ is a post-Lie algebra. The other way round, the enveloping algebra of a post-Lie algebra is a post-associative algebra.

\smallskip

\noindent The \textsl{Grossman--Larson product} on a post-associative algebra is \textcolor{blue}{characterised} by the identity:
\begin{equation}
\label{gl-product}
	(a*b)\rhd c=a\rhd(b\rhd c)
\end{equation}
for any $a,b,c\in B$, in other words $L_{a*b}=L_a\circ L_b$. \textcolor{blue}{It is defined as follows: the map $M:A\to \mop{Der}\mathcal U(A,[-,-])$ defined by $M_ab:=ab+a\rhd b$ is easily seen to verify
$$
	[M_a,M_b]=M_{[\hskip -1pt [a,b]\hskip -1pt]}.
$$
Hence $M$ yields an associative algebra morphism, still denoted by $M$, from $\mathcal U(A,[\![-,-]\!])$ to $\mop{End}\mathcal U(A,[-,-])$. Using for example a Poincar\'e--Birkhoff--Witt basis, one can see that the morphism
\begin{eqnarray*}
	\Theta:\mathcal U(A,[\![-,-]\!])&\longrightarrow &\mathcal U(A,[-,-])\\
	v&\longmapsto& M_v.\mathbf{1}
\end{eqnarray*}
of $\mathcal U(A,[\![-,-]\!])$-modules is bijective, which yields a new associative product $*$ on $\mathcal U(A,[-,-])$ given by $u*v:=\Theta\big(\Theta^{-1}(u).\Theta^{-1}(v)\big)$. Now the identity of $A$ extends to a unique surjective post-associative algebra morphism $\kappa:\mathcal U(A,[-,-])\to\hskip -5pt \to B$. The ideal $\mop{Ker}\kappa$ is stable by both products $.$ and $\rhd$, hence by the product $*$, which thus descends to the quotient $B$.} In particular, for any $a,b,c\in B^1$ we have $a*b=ab+a\rhd b,$ and 
$$ 
	a*b*c=a \cdot b \cdot c+(a\rhd b) \cdot c+b \cdot (a\rhd c)+a \cdot (b\rhd c)+(a \cdot b)\rhd c+(a\rhd b)\rhd c.
$$

An important example of post-Lie algebra is given by $C^\infty(\Cal M,\mathfrak g)$. We suppose that the Lie group $G$, with Lie algebra $\mathfrak g$, acts transitively on the smooth manifold $\Cal M$. Any smooth map $f\in C^\infty(\Cal M,\mathfrak g)$ defines a smooth vector field $\# f$ on $\Cal M$ (i.e., a derivation of $C^\infty(\Cal M)$) via:
\begin{equation}
\label{vector-field}
	\# f(g):=\frac{d}{dt}{\restr{t=0}}\,g\Big(\exp \big(tf(x)\big).x\Big).
\end{equation}
In the language of Lie algebroids, considering the tangent vector bundle and the trivial vector bundle $E=\Cal M\times\mathfrak g$, the map $\#:C^\infty(\Cal M,\mathfrak g)\to\mop{Der} C^\infty(\Cal M)$ is the composition on the left with the anchor map $\rho: E\to T\Cal M$ defined by $\rho (x,X):=\frac{d}{dt}{\restr{t=0}}(\exp tX).x$.\\

Formula \eqref{vector-field} also makes sense for $g\in C^\infty(\Cal M,\mathfrak g)$ or $g\in C^\infty\big(\Cal M,\mathcal U(\mathfrak g)\big)$. It is shown in \cite{MunWri2008} that $C^\infty\big(\Cal M,\mathcal U(\mathfrak g)\big)$, endowed with the pointwise product in $\mathcal U(\mathfrak g)$ as well as the product $\rhd$ given by $f\rhd g:=\# f(g)$ is a post-associative algebra. The Grossman--Larson product reflects the composition of differential operators, in the sense that we have
$$
	\#(f*g)=\#f\circ\#g.
$$
Similarly $C^\infty(\Cal M,\mathfrak g)$, endowed with the pointwise Lie bracket in $\mathfrak g$ and the product $\rhd$ given by $f\rhd g:=\# f(g)$ is a post-Lie algebra.


\subsection{Free post-Lie algebras}
\label{ssect:freepostLie}

It is proven in \cite{MunWri2008} that the free post-associative algebra $\mathcal D_A$ generated by the set $A$ is the algebra of $A$-decorated planar forests endowed with concatenation and left grafting. The latter is defined for any $A$-decorated planar rooted tree $\sigma$ and forest $\tau$:
\begin{equation}
\label{searrow}
	\sigma \rhd \tau = \sum_{v \smop{vertex of}\tau}{\sigma \searrow_{v} \tau},
\end{equation}
where $\sigma \searrow_{v} \tau$ is the decorated forest obtained by grafting the planar tree $\sigma$ on the vertex $v$ of the planar forest $\tau$, such that $\sigma$ becomes the leftmost branch, starting from vertex $v$, of this new tree. It is also well-known that the usual grafting product ``$\to$" given for \textsl{non-planar} rooted trees $\wt\tau$ by the same formula \eqref{searrow}, satisfies the left pre-Lie identity:
$$ 
	\wt\sigma_1 \to ({\wt\sigma}_2 \to \wt\tau) - (\wt\sigma_1 \to {\wt\sigma}_2)\to \wt\tau
	= {\wt\sigma}_2 \to ({\wt\sigma}_1 \to \wt\tau) - ({\wt\sigma}_2 \to \wt\sigma_1)\to \wt\tau.
$$ 
\textcolor{blue}{One the other hand, the linear span of $A$-decorated planar rooted trees endowed with the operation $\rhd$ of \eqref{searrow} is the free magmatic algebra\footnote{The free magma $(\mathcal M_A,*)$ on a set $A$ is the set of well-parenthesised words with letters in $A$. The binary operation $*$ consists in putting each component between an extra pair of parentheses and concatenating them, e.g. $ab*c(de)=(ab)\big(c(de)\big)$. The free magmatic algebra $\langle \mathcal M_A\rangle$ is the vector space freely generated by $\mathcal M_A$, endowed by the bilinear extension of the product $*$. It is determined up to isomorphism by the universal property so that for any vector space $V$ endowed with a bilinear map $\#:V\times V\to V$ and for any set map $f:A\to V$, there is a unique linear map $\overline f:\langle \mathcal M_A\rangle\to V$ respecting both bilinear products.}, see \cite{EM14}}. The definition of $\sigma \rhd \tau$ when $\sigma$ is a decorated forest is given recursively with respect to the number of connected components of $\sigma$, using axiom \eqref {d-alg-induction}.\\

As a result, the free post-Lie algebra $\mathcal P_A$ generated by $A$ is the free Lie algebra generated by the linear span of $A$-decorated planar rooted trees (see also \cite{V2007}).


\subsection{Lie--Butcher series}
\label{ssect:LieButcher}

Let $\Cal M$ be a homogeneous space under the action of a Lie group $G$ with Lie algebra $\mathfrak g$, and let $f:=\{f_i\}_{i\in A}$ be a collection of smooth maps from $\Cal M$ to $\mathfrak g$ indexed by a set $A$. By freeness property, there is a unique post-Lie algebra morphism $\mathcal F_{f}: \mathcal P_A\to C^\infty(\Cal M,\mathfrak g)$ such that $\mathcal F_f(\racine_i)=f_i$. The vector fields $\#\mathcal F_f(\sigma)$, where $\sigma$ is a planar rooted $A$-decorated tree, are the so-called \textsl{elementary differentials}. Similarly, $\mathcal F_f$ extends uniquely to a post-associative algebra morphism  $\mathcal F_f: \mathcal D_A\to C^\infty\big(\Cal M,\mathcal U(\mathfrak g)\big)$. This extended morphism also respects the Grossman--Larson product of both sides.

\medskip

\noindent A \textsl{Lie--Butcher series} is an element of $C^\infty\big(\Cal M,\mathcal U(\mathfrak g)\big)[[h]]$ given by
\begin{equation}
	LB(\alpha,hf):=\sum_{k\ge 0}\ \sum_{\sigma \in F_{A,k}^{\smop{pl}}}h^k\alpha(\sigma)\mathcal F_{hf}(\sigma),
\end{equation}
where $F_{A,k}^{\smop{pl}}$ is the set of $A$-decorated planar rooted forests with $k$ vertices and $\alpha$ is a linear map from $\mathcal P_A$ to the field $\mathbf k$.


\subsection{Three partial orders on planar forests}
\label{ssec:partialorders}

Let $\overline\sigma$ be any planar rooted forest, and $v$, $w$ be two elements in its vertex set $V(\overline\sigma)$. Define a partial order $<$  on $V(\overline\sigma)$ as follows: $v<w$ if there is a path from one root to $w$ passing through $v$. Roots are the minimal elements, and leaves are the maximal elements. \ignore{The rooted forest $\wt{\sigma}$ (discarding its planar structure) is the Hasse diagram of the poset $(V(\overline\sigma),<)$.}\\

Following \cite{A14}, we define a refinement $\ll$ of this order to be the transitive closure of the relation $R$ defined by: $vRw$ if $v<w$, or both $v$ and $w$ are linked to a third vertex $u \in V(\overline\sigma)$, such that $v$ lies on the right of $w$, like this: 
$$
	\arbrebbLab,
$$
or both $v$ and $w$ are roots, with $v$ on the right of $w$.
\vskip 2mm
A further refinement $\lll$ on $V(\overline\sigma)$ is the total order defined as follows: $v \lll w$ if and only if $v$ occurs before $w$ on a path exploring the rooted forest from right to left, starting from the root of the rightmost connected component:
\vskip 6mm
\begin{eqnarray*}
	&\arbreebz&
\end{eqnarray*}
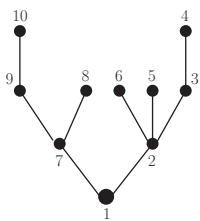
\centerline{\small{A planar rooted tree with its vertices labelled according to total order\ $\!\lll\!$.}}


\subsection{The Hopf algebra of Lie group integrators}
\label{sect:hmkw}

The universal enveloping algebra over the free post-Lie algebra $\mathcal P_A$ endowed with the Grossman--Larson product and deshuffle coproduct\footnote{Uniquely determined by the fact that any $A$-decorated planar rooted tree is primitive.} is a connected Hopf algebra, graded by the number of vertices. Its graded dual is the Hopf algebra of Lie group integrators $\mathcal H_{\smop{MKW}}^A$ introduced by Munthe-Kaas and Wright \cite{MunWri2008}. The convolution product on $\mathcal L(\mathcal H_{\smop{MKW}}^A,\mathbf k)$ is then the Grossman--Larson product naturally extended to series. The product is the shuffle product of planar forests (where the trees are the letters), and the coproduct is given in terms of left-admissible cuts \cite{MunWri2008}:
\begin{equation}
	\Delta(\tau)=\sum_{V'\sqcup V''= V(\tau) \atop V''\ll V'}(\tau\restr{V'})^{\sshu}\otimes \tau\restr{V''}.
\end{equation}
Here $(\tau\restr{V'})^{\sshu}$ is the shuffle product of the connected components of the poset $(V', \ll\restr{V'})$. Note that the restriction of the partial order $\ll$ to $V'$ is generally weaker than the partial order $\ll$ of the forest $\tau\restr {V'}$: the latter makes the poset $V'$ connected, which is generally not the case for the former.

\smallskip

We note that the Hopf algebra of Lie group integrators $\mathcal H_{\smop{MKW}}^A$, endowed with the basis of $A$-decorated planar forests, is a combinatorial Hopf algebra in the sense of Paragraph \ref{sect:cha}.


\section{Planarly branched rough paths}
\label{sec:MKWHA}

We prove in this section the non-degeneracy of the combinatorial Hopf algebra $\mathcal H_{\smop{MKW}}^A$ of Lie group integrators endowed with the basis of $A$-decorated rooted forests, and we define planarly branched rough paths as $\mathcal H_{\smop{MKW}}^A$-rough paths.


\subsection{Tree and forest factorials, volume computations}
\label{ssect:factorials}

For any $s\le t\in \mathbb R$ and any finite poset $P$ we consider the domain
$$
	\Omega^{st}_P:=\big\{(t_v)_{v\in P},\, s\le t_v\le t \hbox{ and } t_v\ge t_w
	\hbox{ for } v < w\big\}\subset\mathbb R^P.
$$
The factorial of the poset $P$ is uniquely determined  by
\begin{equation}
\label{poset-factorial}
	\mop{Volume}(\Omega_{st}^P)=\frac{|t-s|^{|P|}}{P!}.
\end{equation}
For any \textsl{planar} rooted forest $\tau$ (decorated or not), we set \cite{MF17}:
\begin{equation*}
	\tau!:=\big(V(\tau),\ll\big)!
\end{equation*}
In particular, the factorial of a poset is the product of factorials of its connected components. Note, however, that our definition differs from L.~Foissy's definition given in \cite[Definition 33]{F16}. In particular, our notion of factorial is invariant under the canonical involution reversing the partial order, which is not the case for the poset factorial of \cite{F16}.

\medskip

The factorial ${\wt\tau}!$ of a \textsl{non-planar} rooted forest ${\wt\tau}$ is the factorial of the underlying poset $\big(V({\wt\tau}),<\big)$. The notations will always make clear whether a planar or non-planar forest factorial is considered, hence the common notation $-!$ should not cause any confusion.

\begin{lemma}[\cite{MF17}]\label{factorials-volumes}
Let $\wt\tau$ (respectively $\sigma$) be a non-planar (respectively planar) rooted forest. Then:
\begin{enumerate}

\item For any $s \le t \in \mathbb R$, the volume of the domain $\Omega^{st}_{{\wt\tau}}:=\Omega^{st}_{V({\wt\tau}),<}$ is equal to $\frac{1}{{\wt\tau}!}(t-s)^{\vert {\wt\tau}\vert}$.

\item \label{deux} For any $s \le t \in \mathbb R$, the volume of the domain $\Omega^{st}_{V(\sigma),\ll}$
is equal to $\frac{1}{\sigma !}(t-s)^{\vert \sigma\vert}$.

\item \label{inv-fac-functorial} The following identity holds:
\begin{equation*}
	\frac{1}{{\wt\tau}!} = \mop{Sym}(\wt\tau)\sum_{\sigma\surj{4}{\wt\tau}}\frac{1}{\sigma !},
\end{equation*}
where the sum runs over all the planar representatives $\sigma$ of ${\wt\tau}$, and where $\mop{Sym}(\sigma)$ is the symmetry factor of the planar rooted forest $\sigma$.
\end{enumerate}
\end{lemma}

\begin{proof}
The volume of $\Omega_{{\wt\tau}}^{st}$ is multiplicative, i.e., it is the product of the volumes of $\Omega_c^{st}$ where $c$ runs over the connected components of ${\wt\tau}$. The inverse of the forest factorial shares the same property. Hence it is sufficient to check the result on trees. We proceed by induction on the number of vertices. The case of one vertex boils down to:
$$
	\mop{Vol}(\Omega_{\racine}^{st})=t-s=\frac{1}{\racine!}(t-s).
$$
Suppose that ${\wt\tau}=B^+(f)$ is a tree with at least two vertices. From \eqref{poset-factorial} we get:
$$
	\frac{1}{{\wt\tau}!}=\frac{1}{\vert {\wt\tau}\vert}\frac{1}{f!}.
$$
On the other hand, using the induction hypothesis, we have
\begin{eqnarray*}
	\mop{Vol}(\Omega_{{\wt\tau}}^{st})
	&=&\int_{s}^t\mop{Vol}(\Omega_{f}^{sz})\,dz\\
	&=&\frac{1}{f!}\int_{s}^t (z-s)^{\vert f\vert}\, dz\\
	&=&\frac{1}{f!}\frac{1}{\vert {\wt\tau}\vert}(t-s)^{\vert {\wt\tau}\vert}\\
	&=&\frac{1}{{\wt\tau}!}(t-s)^{\vert {\wt\tau}\vert}.
\end{eqnarray*}
Now let $\sigma$ be a planar rooted forest. The volume of $\Omega_{V(\sigma),\ll}^{st}$ is multiplicative for the shuffle product but not for the concatenation. However, any planar rooted forest $\sigma$ admits a natural unique decomposition:
\begin{equation*}
	\sigma=\sigma'\times\sigma''=\sigma'B^+(\sigma''),
\end{equation*}
where $\sigma'$ and $\sigma''$ are again (possibly empty) planar rooted forests. The poset $(V(\sigma),\ll)$ is obtained by considering the direct product of the two posets $(V(\sigma'),\ll)$ and $(V(\sigma''),\ll)$, and adding an extra element (the root of $\sigma)$ smaller than any other element. From \eqref{poset-factorial} again we get:
\begin{equation}
\label{rec-fac-one}
	\sigma !=\vert\sigma\vert\sigma'!\sigma''!
\end{equation}
The second assertion is then proved recursively with a computation analogous to one above:
\begin{eqnarray*}
	\mop{Vol}(\Omega_{V(\sigma),\ll}^{st})
	&=&\int_{z=s}^t\mop{Vol}(\Omega_{V(\sigma'),\ll}^{sz})\mop{Vol}(\Omega_{V(\sigma''),\ll}^{sz})\,dz\\
	&=&\frac{1}{\sigma'!\sigma''!}\int_{z=s}^t (s-z)^{\vert \sigma'\vert+\vert\sigma''\vert}\, dz\\
	&=&\frac{1}{\vert\sigma\vert\sigma'!\sigma''!}(t-s)^{\vert \sigma\vert}\\
	&=&\frac{1}{\sigma !}(t-s)^{\vert {\sigma}\vert}.
\end{eqnarray*}

Now let $\wt\tau$ be a non-planar rooted forest. For any $v \in V(\wt\tau)$, let $\mop{St}(v)$ be the set of vertices immediately above $v$, let $S_v$ be the set of total orders on $\mop{St}(v)$, and finally let $S_{\wt\tau}$ be the product of the sets $S_v$ for $v \in V(\wt\tau)$. Any element $\prec\ \in S_{\wt\tau}$ obviously defines a binary relation on $V(\wt\tau)$, also denoted by $\prec$: to be precise, $w' \prec w''$ if and only if there exists $v\in V(\wt\tau)$ such that $w',w''\in\mop{St}_v$ and $w'\prec w''$ inside $\mop{St}_v$. For any element $\prec\ \in S_{\wt\tau}$, let $\Cal R_\prec$ be the binary relation on $V(\wt\tau)$ defined by:
\begin{equation*}
	w'\Cal R w'' \hbox{ if and only if } w'<w'' \hbox{ or } w'\prec w'',
\end{equation*}
and let $\ll_\prec$ be the transitive closure of $\Cal R_\prec$. This is a partial order refining the forest order $<$, which, by ordering the branches at any vertex, defines a unique planar representative $\sigma_\prec$ of $\wt\tau$. 

The third assertion comes then from the following fact: the domain $\Omega_{{\wt\tau}}^{st}$ is the union of the domains $\Omega_{V(\wt\tau),\ll_\prec}^{st}$ (mutually disjoint apart from a Lebesgue-negligible intersection) where $\prec$ runs over $S_{\wt\tau}$. Now two elements $\prec$ and $\prec'$ give rise to the same planar representative $\sigma$ if and only if the unique permutation of $V(\tau)$ which induces an increasing map from $(\mop{St}(v),\prec)$ onto $(\mop{St}(v),\prec')$ is an automorphism of $\wt \tau$.
\end{proof}

\noindent As a consequence, we easily obtain an analogue of Lemma 4.4 in reference \cite{Gubi2010}:

\begin{corollary}\label{planar-forest-binomial}
For any rooted planar forest $\tau$ and for any $h,k \ge 0$ the following holds:
\begin{equation*}
	(h+k)^{\vert\tau\vert} = \sum_{V' \sqcup V'' = V(\tau),\atop V''\ll V'}\frac{\tau !}
	{(V',\ll\restr{V'})!\,\tau\restr{V''}!}h^{\vert V'\vert}k^{\vert V''\vert}.
\end{equation*}
\end{corollary}

\begin{proof}
Let $s\le u\le t$ three real numbers, with $u-s=h$ and $t-u=k$. From the definition of the domain $\Omega_{\tau,\ll}^{st}$ we immediately can express it as the following union with Lebesgue-negligible pairwise intersections:
\begin{equation*}
	\Omega_{V(\tau),\ll}^{st}=\bigcup_{V(\tau)=V'\sqcup V'',\atop V''\ll V'}
	\Omega^{su}_{(V',\ll\srestr{V'})}\times\Omega^{ut}_{(V'',\ll\srestr{V''})}.
\end{equation*}
The conclusion then follows from item \ref{deux} of Lemma \ref{factorials-volumes}.
\end{proof}

\begin{remark}\label{gen-posets}\rm 
The inversion of the order in the definition of $\Omega_P^{st}$ is not really necessary, as this inversion amounts to a change of variables $t_v\mapsto s+t-t_v$, which does not change the volume. But it makes the proof slightly more direct.
\end{remark}

\begin{remark}\rm
Applying Corollary \ref{planar-forest-binomial} in the special case $h=k=1$ shows that $\tau\mapsto 1/\tau!$ extends linearly to the unique inverse-factorial character of the Hopf algebra $\mathcal H_{\smop{MKW}}^A$ taking value $1$ on the letters of $A$. As a consequence, the combinatorial Hopf algebra $\mathcal H^A_{\smop{MKW}}$ endowed with the decorated forest basis is non-degenerate. The analogue is true for non-planar forests and the $A$-decorated Butcher--Connes--Kreimer Hopf algebra, due to Lemma 4.4 in \cite{Gubi2010}. As a consequence, assertion \eqref{inv-fac-functorial} of Lemma \ref{factorials-volumes} can be derived in a purely algebraic way, using the Hopf algebra morphism $\Omega: \mathcal H_{\smop{BCK}}^A \to \mathcal H_{\smop{MKW}}^A$ of \cite{MunWri2008} given by  \eqref{omega} below, as well as the functoriality of inverse-factorial characters.
\end{remark}
\textcolor{blue}
{
\begin{remark}\rm
Another interesting example comes from the extraction-contraction Hopf algebra $\mathcal H^A_{\smop{EC}}$ of reference \cite{CEM2011}, in which the grading is given by the number of edges. In view of the recursive definition of the inverse factorial character given in Paragraph \ref{sect:inv-fact}, and due to the fact that the coproduct of a forest with $n$ edges contains exactly $2^n$ elements, the inverse factorial character is identically equal to $1$ on any forest. It is however not clear whether one can consider the rather exotic corresponding notion of rough path as a driving object for some kind of rough differential equation.
\end{remark}
}
\noindent By a straightforward iteration of \eqref{rec-fac-one} one obtains another recursive formula for the planar factorial:

\begin{proposition}
Let $\sigma=\sigma_1\cdots \sigma_k$ be a planar forest, decorated or not, with connected components $\sigma_j=B_+(\tau_j),\,j=1,\ldots k$. Then
\begin{equation}
\label{rec-fac-two}
	\sigma!=|\sigma_1|.|\sigma_1\sigma_2|\cdots |\sigma_1\cdots\sigma_k|\tau_1!\cdots\tau_k!
\end{equation}
\end{proposition}


\subsection{Planarly branched rough paths}
\label{sect:pbrp-def}

The notion of planarly branched rough paths is given in the next definition. It is motivated from a Hopf algebraic point of view. Its significance  for controlled rough differential equations will become clear further below.

\begin{definition}\label{def:planarBRP}
Let $\gamma\in]0,1]$ and let $A$ be a finite alphabet. A $\gamma$-regular planarly branched rough path is a $\gamma$-regular $\mathcal H_{\smop{MKW}}^A$-rough path.
\end{definition}


\section{Simple and contracting arborification in the planar setting}
\label{sec:quotient}


\subsection{A projection onto the shuffle Hopf algebra: planar arborification}
\label{sect:planar-arb}

Let $A$ be an alphabet, and let $\Cal H_{\sshu}^A$ (resp.~$\Cal H_{\smop{MKW}}^A$) be the shuffle Hopf algebra with letters in $A$ (resp.~the Hopf algebra of $A$-decorated planar forests). The \textsl{planar arborification map} $\frak a_{\ll}:\Cal H_{\smop{MKW}}^A\to\Cal H_{\sshu}^A$ sends any planar decorated forest to the sum of its linear extensions. It is defined for any degree $n$ planar $A$-decorated forest $\tau=(\overline\tau,\varphi)$ as follows:
\begin{equation*}
	\frak a_{\ll}(\overline\tau,\varphi):=\sum_{\alpha:(\Cal V(\overline\tau),\ll)\nearrow \{1,\ldots,n\}}
	\varphi\circ\alpha^{-1}(1)\cdots \varphi\circ\alpha^{-1}(n),
\end{equation*}
where the sum runs over the increasing bijections from the poset $\big(\Cal V(\overline\tau),\ll)$ onto $\{1,\ldots,n\}$. 
As an example, we have:
\begin{equation*}
	\frak a_{\ll}(\arbreadec ab\ \racine^c)=bac,\hskip 12mm \frak a_{\ll}(\racine^c\arbreadec ab\;)=bca+cba.
\end{equation*}

This definition is directly inspired from the simple arborification map $\frak a:\Cal H_{\smop{BCK}}^A\to\Cal H_{\sshu}^A$ which sends any (non-planar) decorated forest to the sum of its linear extensions \cite{E92, FM17, F02}. It is defined for any degree $n$ non-planar $A$-decorated forest by:
\begin{equation*}
	\frak a(f,\varphi):=\sum_{\alpha:(\Cal V(f),<)\nearrow \{1,\ldots,n\} }
	\varphi\circ\alpha^{-1}(1)\cdots \varphi\circ\alpha^{-1}(n),
\end{equation*}
where the sum runs over the increasing bijections from the poset $\big(\Cal V(f),<)$ onto $\{1,\ldots,n\}$. 

\begin{lemma}\label{planar-decomp} (Canonical decomposition of planar forests)
\begin{enumerate}
\item Any non-empty $A$-decorated planar forest $\tau$ admits a unique decomposition:
\begin{equation*}
	\tau=\tau'\times_a \tau'',
\end{equation*}
where $\tau'$ and $\tau''$ are $A$-decorated planar forests, $a \in A$ and $\tau'\times_a \tau''$ stands for $\tau'B_a^+(\tau'')$.

\item The planar arborification map can be recursively defined by $\frak a_{\ll}(\bm 1)=\bm 1$ and 
\begin{equation*}
	\frak a_{\ll}(\tau'\times_a \tau'')=[\frak a_{\ll}(\tau')\shu\frak a_{\ll}(\tau'')]a.
\end{equation*}
\end{enumerate}
\end{lemma}

\begin{proof}
The first assertion is straightforward, the second is a direct consequence of the poset structure of $V(\tau'\times_a \tau'')=V(\tau')\sqcup  V(\tau'')\sqcup \{a\}$ under the partial order $\ll$, which is entirely determined by the fact that 
\begin{enumerate}
\item the restriction of $\ll$ to $V(\tau')$ is the partial order $\ll$ determined by the planar forest $\tau'$, and similarly for $\tau''$,

\item $a \in A$ is the unique minimum,

\item vertices of $\tau'$ are incomparable with vertices of $\tau''$.
\end{enumerate}
\end{proof}

\begin{theorem}\label{planar-arb}
The planar arborification map $\frak a_{\ll}$ is a surjective Hopf algebra morphism, combinatorial if the alphabet $A$ is finite, and the diagram below commutes.
\diagramme{
\xymatrix{
\Cal H_{\smop{BCK}}^A \ar[rr]^\Omega \ar@{>>}[ddrr]_{\frak a} && \Cal H_{\smop{MKW}}^A\ar@{>>}[dd]^{\frak a_{\ll}}\\
&&\\
&&\Cal H_{\sshu}^A
}
}
\noindent where $\Omega$ is the symmetrization map \cite[Definition 8]{MunWri2008}.
\end{theorem}

\begin{proof}
It is well-known that $\frak a$ and $\Omega$ are Hopf algebra morphisms \cite{F02,MunWri2008}. The map $\Omega$ is given by:
\begin{equation}
\label{omega}
	{\Omega(f)={\mop{Sym}}(f)\sum_{\tau\to\hskip -6.5pt\to f}\tau},
\end{equation}
from which the commutation of the diagram easily follows. It only remains to prove by direct checking that $\frak a_{\ll}$ respects the Hopf algebra structures. For any $A$-decorated planar forests $\tau,\omega$ which admit canonical decompositions $\tau=\tau'\times_a \tau''$ and $\omega=\omega'\times_b \omega''$ according to Lemma \ref{planar-decomp}, we compute, using induction on the sum of degrees $\vert \tau\vert +\vert \omega\vert$:
\allowdisplaybreaks
\begin{eqnarray*}
	\frak a_{\ll}(\tau\shu \omega)
		&=&\frak a_{\ll}\big((\tau'\times_a \tau'')\shu (\omega'\times_b \omega'')\big)\\
		&=&\frak a_{\ll}\big(\tau'B^+_a(\tau'')\shu \omega' B^+_b(\omega'')\big)\\
		&=&\frak a_{\ll}\Big([\tau'\shu \omega'B^+_b(\omega'')]B^+_a(\tau'')
			+[\tau'B^+_a(\tau'')\shu \omega']B^+_b(\omega'')\Big)\\
	&=&\frak a_{\ll}\Big([\tau'\shu(\omega'\times_b \omega'')]\times_a \tau''
			+[(\tau'\times_a\tau'')\shu \omega']\times_b \omega''\Big)\\
	&=&\Big[\frak a_{\ll}\big(\tau'\shu(\omega'\times_b \omega'')\big)\shu \frak a_{\ll}(\tau'')\Big]a
	+\Big[\frak a_{\ll}\big((\tau'\times_a \tau'')\shu \omega'\big)\shu\frak a_{\ll}(\omega'')\Big]b\\
	&=&\Big[\frak a_{\ll}(\tau')\shu\frak a_{\ll}(\omega'\times_b \omega'')\shu\frak a_{\ll}(\tau'')\Big]a
	+\Big[\frak a_{\ll}(\tau'\times_a \tau'')\shu\frak a_{\ll}(\omega')\shu\frak a_{\ll}(\omega'')\Big]b\\
	&=&\Big[\frak a_{\ll}(\tau')\shu\big[\frak a_{\ll}(\omega')\shu \frak a_{\ll}(\omega'')\big]b\shu\frak a_{\ll}(\tau'')\Big]a
	+\Big[\big[\frak a_{\ll}(\tau')\shu\frak a_{\ll} (\tau'')\big]a\shu\frak a_{\ll}(\omega')\shu\frak a_{\ll}(\omega'')\Big]b\\
	&=&\Big[\frak a_{\ll}(\tau')\shu\frak a_{\ll}(\tau'')\Big]a\,\shu\,\Big[\frak a_{\ll}(\omega')\shu\frak a_{\ll}(\omega'')\Big]b\\
	&=&\frak a_{\ll}(\tau)\shu\frak a_{\ll}(\omega).
\end{eqnarray*}
To check compatibility with coproducts, we introduce the linear operator of left concatentation, $L_a:\Cal H_{\sshu}^A\to\Cal H_{\sshu}^A$, defined by $L_a(w)=wa$ for any word $w\in A^*$. It clearly verifies:
\begin{equation*}
	L_a\circ\frak a_{\ll}=\frak a_{\ll}\circ B^+_a.
\end{equation*}
\noindent For any $A$-decorated planar forest $\tau=\tau'\times_a \tau''$ we compute, using induction on the degree $\vert \tau\vert$:
\begin{eqnarray*}
	\Delta\frak a_{\ll}(\tau)
	&=&\Delta\Big(\big[\frak a_{\ll}(\tau')\shu\frak a_{\ll}(\tau'')\big]a\Big)\\
	&=&(\mop{Id}\otimes L_a)\Delta\Big(\frak a_{\ll}(\tau')\shu\frak a_{\ll}(\tau'')\Big)
	+\frak a_{\ll}(\tau)\otimes\bm 1\\
	&=&(\mop{Id}\otimes L_a)\Big(\Delta\frak a_{\ll}(\tau')\shu\Delta\frak a_{\ll}(\tau'')\Big)
	+\frak a_{\ll}(\tau)\otimes\bm 1\\
	&=&(\mop{Id}\otimes L_a)\Big((\frak a_{\ll}\otimes\frak a_{\ll})\Delta \tau'\shu(\frak a_{\ll}\otimes\frak a_{\ll})
	\Delta \tau''\Big)+\frak a_{\ll}(\tau)\otimes\bm 1\\
	&=&(\frak a_{\ll}\otimes\frak a_{\ll})\Big((\mop{Id}\otimes B^+_a)(\Delta \tau'\shu\Delta \tau'')
	+\tau\otimes\bm 1\Big)\\
	&=&(\frak a_{\ll}\otimes\frak a_{\ll})\Big((\mop{Id}\otimes B^+_a)\big(\Delta (\tau'\shu \tau'')\big)
	+\tau\otimes\bm 1\Big)\\
	&=&(\frak a_{\ll}\otimes\frak a_{\ll})(\Delta \tau).
\end{eqnarray*}
Compatibility with units and co-units is immediate, and compatibility with antipodes comes for free due to connectedness of both Hopf algebras.
\end{proof}

\begin{proposition}\label{rp-arbo}
Let $\gamma\in]0,1]$, let $A$ be a finite alphabet with $d$ letters, and let $\mathbb X_{st}$ be a $\gamma$-regular rough path in the classical sense on $\mathbb R^d$. Then its arborified version $\wt{\mathbb X}_{st}:=\mathbb X_{st}\circ\mathfrak a^\ll$ is a $\gamma$-regular planarly branched rough path on $\mathbb R^d$.
\end{proposition}

\begin{proof}
This is an immediate consequence of Proposition \ref{funct-rough}.
\end{proof}

\noindent Proposition \ref{rp-arbo} calls for the following definition:

\begin{definition}
A $\gamma$-regular planarly branched rough path $\mathbb Z_{st}$ on $\mathbb R^d$ is geometric if there exists a $\gamma$-regular rough path $\mathbb X_{st}$ in the classical sense such that $\mathbb Z_{st}$ is its arborified version, i.e.
$$
	\mathbb Z_{st}=\mathbb X_{st}\circ\mathfrak a^\ll.
$$
\end{definition}
\textcolor{blue}
{
\begin{remark}\rm
Any geometric branched rough path is then geometric by definition. The converse is true at the price of inflating the alphabet, as in the branched case \cite[Paragraph 4.2]{HaiKel2015}. Indeed, the Hopf algebra $\Cal H_{\smop{MKW}}^A$ is, as an algebra, the shuffle algebra of the set of $A$-decorated planar rooted trees. Any planarly branched rough path is then geometric provided this bigger alphabet is considered.
\end{remark}
}


\subsection{Planar contracting arborification}
\label{sect:planar-arb-c}

We present a contracting version of planar arborification which has some interest in its own right, although it will not be directly used in the present paper. Suppose that the alphabet $A$ carries an Abelian semigroup structure $(a,b)\mapsto [a+b]$. The quasi-shuffle Hopf algebra is isomorphic to $\Cal H_{\sshu}^A$ as coalgebra. The quasi-shuffle product is recursively defined by $\emptyset\qshu w=w\qshu\emptyset=w$ for any word $w\in A^*$ and:
\begin{equation*}
	av\qshu bw=a(v\qshu bw)+b(av\qshu w)+[a+b](v\qshu w)
\end{equation*}
for any letters $a,b\in A$ and words $(v,w)\in A^*$. For example, $a\qshu b=ab+ba+[a+b]$, and $ab\qshu c=abc+acb+cab+[a+c]b+a[b+c]$. It is well-known \cite{H00} that the quasi-shuffle product together with deconcatenation give rise to a Hopf algebra $\Cal H_{\sqshu}^A$ isomorphic to the shuffle Hopf algebra $\Cal H_{\sshu}^A$.

\medskip

The \textsl{planar contracting arborification map} $\frak a^c_{\ll}:\Cal H_{\smop{MKW}}^A\to\Cal H_{\sqshu}^A$ sends any planar decorated forest to the sum of its linear extensions \textsl{including contraction terms}. It is defined for any degree $n$ planar $A$-decorated forest as follows:
\begin{equation*}
	\frak a_{\ll}(\tau,\varphi):=\sum_{r\ge 0}\ \sum_{\alpha:(V(\tau),\ll)\scto 
	\{1,\ldots,n-r\}}\varphi\circ\alpha^{-1}(1)\cdots \varphi\circ\alpha^{-1}(n-r)
\end{equation*}
where the inner sum runs over the increasing surjections from the poset $\big(V(\tau),\ll)$ onto $\{1,\ldots,n-r\}$, i.e., surjective maps $\alpha$ such that $u\ll u'\in V(\tau)$ and $u\neq u'$ implies $\alpha(u)<\alpha(u')$. It can happen that $\alpha^{-1}(j)$ contains several terms: in that case, $\varphi\circ \alpha^{-1}(j)$ is to be understood as the sum in $A$ of the terms $\varphi(u),\,u\in\alpha^{-1}(j)$. As an example, we have:
\begin{equation*}
	\frak a^c_{\ll}(\arbreadec ab\ \racine^c)
	=bac,\hskip 12mm \frak a^c_{\ll}(\racine^c\arbreadec ab\;)=bca+cba+[b+c]a.
\end{equation*}

This definition is directly inspired from the contracting arborification map $\frak a^c:\Cal H_{\smop{BCK}}^A\to\Cal H_{\sqshu}^A$ which sends any (non-planar) decorated forest to the sum of its linear extensions including contraction terms \cite{EFM17,E92, FM17, F02}. It is defined for any degree $n$ non-planar $A$-decorated forest by:
\begin{equation*}
	\frak a^c(f,\varphi):=\sum_{r\ge 0}\ \sum_{\alpha:(V(f),<)\scto \{1,\ldots,n-r\} }
	\varphi\circ\alpha^{-1}(1)\cdots \varphi\circ\alpha^{-1}(n)
\end{equation*}
where the inner sum runs over the increasing surjections from the poset $\big(V(f),<)$ onto $\{1,\ldots,n-r\}$, and is a surjective Hopf algebra morphism from $\Cal H_{\smop{BCK}}^A$ onto $\Cal H^A_{\sqshu}$. 

An analogue of Theorem \ref{planar-arb} holds:
\begin{theorem}\label{planar-arb-c} $\ $
\begin{enumerate}

\item The planar contracting arborification map can be recursively defined by $\frak a^c_{\ll}(\bm 1)=\bm 1$ and
{\rm
\begin{equation*}
	\frak a^c_{\ll}(\tau'\times_a \tau'')=[\frak a^c_{\ll}(\tau')\qshu\frak a^c_{\ll}(\tau'')]a.
\end{equation*}
}
\item The planar contracting arborification map $\frak a^c_{\ll}$ is a surjective Hopf algebra morphism, combinatorial if the alphabet $A$ is finite, and the diagram below commutes.
{\rm
\diagramme{
\xymatrix{
\Cal H_{\smop{BCK}}^A \ar[rr]^\Omega \ar@{>>}[ddrr]_{\frak a^c} && \Cal H_{\smop{MKW}}^A\ar@{>>}[dd]^{\frak a_{\ll}^c}\\
&&\\
&&\Cal H_{\sqshu}^A
}
}
}
\end{enumerate}
\end{theorem}

\begin{proof}
Entirely similar to proof of the analogous results on planar arborification on Paragraph \ref{sect:planar-arb}. Details are left to the reader.
\end{proof}


\section{Rough differential equations on homogeneous spaces}
\label{sec:RDEHom}

In this section, we prove the convergence of the formal solutions of the rough differential equation \eqref{eq:control1} under particular analyticity assumptions.


\subsection{Formal solutions of a rough differential equation on a homogeneous space}

Let $t\mapsto X_t:=(X_t^1,\ldots,X_t^d)$ be a differentiable path with values in $\mathbb R^d$. Let $A=\{a_1,\ldots,a_d\}$ be an alphabet with $d$ letters. The controlled differential equation we are looking at writes:
\begin{equation}
\label{rde-rappel}
	dY_{st}=\sum_{i=1}^d \# f_i(Y_{st})\,dX_t^i
\end{equation}
with initial condition $Y_{ss}=y$. The unknown is a path $Y_s: t \mapsto Y_{st}$ in a homogeneous space $\Cal M$, with transitive action $(g,y)\mapsto g.y$ of a Lie group $G$ on it. The elements in $f:=\{f_i\}_{i=1}^d$ are smooth maps from $\Cal M$ into the Lie algebra $\frak g=\mop{Lie}(G)$, which in turn define smooth vector fields $y \mapsto \# f_i(y)$ on $\Cal M$:
\begin{equation}
\label{vectorfield-rappel}
	\# f_i(y):=\frac{d}{dt}{\restr{t=0}}\exp \big(tf_i(y)\big).y \in T_y \Cal M.
\end{equation}
It has been explained in the Introduction how Equation \eqref{rde-rappel} is lifted to the following differential equation with unknown $\bm Y_{st}\in C^\infty\big(\mathcal M,\mathcal U(\mathfrak g)\big)[[h]]$ and step size $h=t-s$:
\begin{equation}
\label{rde-lifted-rappel}
	d\bm Y_{st} = \sum_{i=1}^d \bm Y_{st}*f_i\,dX_t^i
\end{equation}
with initial condition $\bm Y_{ss}=\bm 1$. Recall that the $\ast$ product stands for the Grossman--Larson product in the post-associative algebra $C^\infty\big(\mathcal M,\mathcal U(\mathfrak g)\big)$. The formal solution of \eqref{rde-rappel} is recovered by
\begin{equation}
\label{recover-rde}
	\psi(Y_{st})=(\#\bm Y_{st}.\psi)(y).
\end{equation}
for any test function $\psi\in C^\infty(\mathcal M)$. A further step in abstraction leads to the fundamental differential equation in the character group of the Hopf algebra $\mathcal H^A_{\smop{MKW}}$:
\begin{equation}
\label{rde-postlifted-rappel}
	d\mathbb Y_{st} = \sum_{i=1}^d \mathbb Y_{st}*\racine_i\,dX_t^i
\end{equation}
with initial condition $\mathbb Y_{ss}=\bm 1$. The $\ast$ product now stands for the Grossman--Larson product in the completed free post-associative algebra $\wh{\mathcal D_A}$ generated by $A$. The solution of \eqref{rde-lifted-rappel} then is obtained by $\mathbf Y_{st}=\mathcal F_f(\mathbb Y_{st})$, where $\mathcal F_f$ is the $h$-adic completion (with $h=t-s$) of the unique post-associative algebra morphism from $\mathcal D_A$ to $C^\infty\big(\mathcal M,\mathcal U(\mathfrak g)\big)$ which sends $\racine_j$ to $f_j$. By using the integral formulation and Picard iteration, the solution of \eqref{rde-postlifted-rappel} is given by the word series expansion:
\begin{equation}
\label{expansion-word-rappel}
	\mathbb Y_{st}=\sum_{\ell\ge 0}\sum_{w=a_{i_1}\cdots a_{i_\ell}\in A^*}\langle\bx_{st},w\rangle \racine_{i_\ell}*\cdots*\racine_{i_1}.
\end{equation}

\begin{theorem}[Planar arborification-coarborification transform]\label{planar-arbo-coarbo}
The solution of \eqref{rde-postlifted-rappel} is given by the following expansion indexed by $A$-decorated planar rooted forests:
\begin{equation}
\label{arbo-coarbo-pl}
	\mathbb Y_{st}=\sum_{\tau\in F^A_{\smop{pl}}}\langle \mathbb X_{st}\circ\mathfrak a^{\ll},\,\tau\rangle\tau.
\end{equation}
\end{theorem}

\begin{proof}
For any $\overline\tau\in F_{\smop{pl}}$, let $\mathcal L(\overline\tau)$ be the set of linear extensions of $\overline\tau$, i.e., the set of total orders $\prec$ on $V(\overline\tau)$ compatible with the partial order $\ll$, i.e., such that $u\ll v\Rightarrow u\prec v$ for any $u,v\in V(\overline\tau)$. Now let $\tau=(\overline\tau,\alpha)$ be an $A$-decorated forest, and let $\tau_\prec$ be the word in $A^*$ obtained from $(\overline\tau,\alpha)$ by displaying the decorations of the vertices of $\overline\tau$ from left to right according to the total order $\prec$. We will use the notation $\mathcal L(\tau)$ instead of $\mathcal L(\overline\tau)$. It can be easily shown that the planar arborification admits the following explicit expression:
\begin{equation*}
	\mathfrak a^{\ll}(\tau)=\sum_{\prec\in\mathcal L(\tau)} \tau_\prec.
\end{equation*}
The following lemma is easily proven by induction on the length:

\begin{lemma}\label{lemma-iterated-gl}
For any word $w=a_{i_1}\cdots a_{i_n}\in A^*$, we have:
\begin{equation}
\label{iterated-gl}
	\racine_{i_n}*\cdots *\racine_{i_1}
	=\sum_{\overline\tau\in T^{[n]}_{\smop{pl}}}\ \sum_{\prec\in\mathcal L(\overline\tau)}(\overline\tau,\alpha_\prec),
\end{equation}
where $\alpha_\prec:V(\overline\tau)\to A$ is the decoration map which sends the $j$-th vertex to $a_{i_j}$ according to $\prec$.
\end{lemma}

\noindent The total number of terms is $n!$. For example we have
$$
	\racine_j*\racine_i=\racine_j\racine_i+\arbrea_i^j,\hskip 12mm
	\racine_k*\racine_j*\racine_i=\racine_k\racine_j\racine_i+
	\arbrea_j^k\racine_i+ 
	\arbrebb_i^{\,j\hskip -7mm k}+
	{\arbreba_i^j}^{\hskip -4.5pt k}+ 
	\racine_j\arbrea_i^k+
	\racine_k\arbrea_i^j.
$$
We compute, using Lemma \ref{lemma-iterated-gl}:
\begin{eqnarray*}
	\mathbb Y_{st}
	&=&\sum_{\ell\ge 0}\sum_{w=a_{i_1}\cdots a_{i_\ell}\in A^*}\langle\bx_{st},w\rangle \racine_{i_\ell}*\cdots*\racine_{i_1}\\
	&=&\sum_{\ell\ge 0}\sum_{w=a_{i_1}\cdots a_{i_\ell}\in A^*}\langle\bx_{st},w\rangle\sum_{\overline\tau \in T^{[n]}_{\smop{pl}}}\ \sum_{\prec\in\mathcal L(\overline\tau)}(\overline\tau,\alpha_\prec)\\
	&=&\sum_{\tau \in T^A_{\smop{pl}}}\sum_{\prec\in\mathcal L(\tau)}\langle\mathbb X_{st},\tau_\prec\rangle\,\tau\\
	&=&\sum_{\tau \in T^A_{\smop{pl}}}\langle\mathbb X_{st}\circ\mathfrak a^{\ll},\tau\rangle\tau.
\end{eqnarray*}
\end{proof}
\noindent Theorem \ref{planar-arbo-coarbo} calls for the following definition.

\begin{definition}
Let $\gamma\in ]0,1]$ and let $t\mapsto X_t:=(X_t^1,\ldots,X_t^d)$ be a $\gamma$-H\" older continuous path with values in $\mathbb R^d$. A formal solution of Equation \eqref{rde-rappel} driven by $X$ is defined by
\begin{equation}
	Y_{st}=\mathcal \#\mathcal F_f(\mathbb Y_{st})(y)
\end{equation}
where $\mathbb Y_{st}$ is given by the expansion
\begin{equation}\label{arbo-expansion}
	\mathbb Y_{st}=\sum_{\tau\in F^A_{\smop{pl}}}\langle \wt{\mathbb X}_{st},\,\tau\rangle\tau
\end{equation}
where $\wt {\mathbb X}_{st}$ is any $\gamma$-regular planarly branched rough path such that $\langle \wt X_{st},\racine_j\rangle=X_t^j-X_s^j$ for any $j\in\{1,\ldots, d\}$.
\end{definition}
We will freely identify the planarly branched rough path $\wt{\mathbb X}_{st}$ with the expansion $\mathbb Y_{st}$ as grouplike elements of the dual of $\mathcal H_{\smop{MKW}}^A$.


\subsection{Cauchy estimates}
\label{ssect:chauchy}

We borrow material from \cite{E92}, see also \cite{FM17, CMP}, adapting it to general homogeneous spaces. For any compact neighbourhood $\mathcal U$ of the origin in $\mathbb C^n$, let $\mathcal A_{\mathcal U}$ be the subspace of analytic germs defined on $\mathcal U$. We have precisely
\begin{equation*}
	\mathcal A_{\mathcal U}=\{\varphi,\, \|\varphi\|_{\mathcal U}<+\infty\},
\end{equation*}
with the norm
\begin{equation*}
	\|\varphi\|_{\mathcal U}:=\mop{sup}_{y\in\mathcal U}|\varphi(y)|
\end{equation*}
making $\mathcal A_{\mathcal U}$ a Banach space. Now let $\mathcal V$ be another compact neighbourhood of the origin such that $\mathcal V\subset\mathring{\mathcal U}$. We consider the operator norm defined for any linear operator $P:\mathcal A_{\mathcal U}\to\mathcal A_{\mathcal U}$ by
\begin{equation*}
	\|P\|_{\mathcal U,\mathcal V}
	=\sup_{\varphi\in\mathcal A_{\mathcal U}-\{0\}}\frac{\|P\varphi\|_{\mathcal V}}{\|\varphi\|_{\mathcal U}}.
\end{equation*}

\noindent The two following lemmas are straightforward.

\begin{lemma}\label{easy-lemma-one}
Let $0\in\mathcal V\subset\mathring{\mathcal U}\subset\mathcal U$ be two compact neighbourhoods of the origin, and let $f\in\mathcal A_{\mathcal U}$. Let $P:\mathcal A_{\mathcal U}\to\mathcal A_{\mathcal U}$ be a linear operator. Denoting by $f:\mathcal A_{\mathcal U}\to\mathcal A_{\mathcal U}$ the pointwise multiplication operator by $f$, then the following estimate holds:
\begin{equation*}
	\|fP\|_{\mathcal U,\mathcal V}\le \|f\|_{\mathcal V}\|P\|_{\mathcal U,\mathcal V}.
\end{equation*}
\end{lemma}

\begin{lemma}\label{easy-lemma-two}
Let $0\in\mathcal V\subset\mathring{\mathcal W}\subset\mathcal W\subset\mathring{\mathcal U}\subset\mathcal U$ be three compact neighbourhoods of the origin, and let $P,Q:\mathcal A_{\mathcal U}\to\mathcal A_{\mathcal U}$ be a two linear operators. Then we have:
\begin{equation*}
	\|P\circ Q\|_{\mathcal U,\mathcal V}\le \|P\|_{\mathcal W,\mathcal V}\|Q\|_{\mathcal U,\mathcal W}.
\end{equation*}
\end{lemma}

\begin{proposition}\label{estimate-vf}
Let $0\in\mathcal V\subset\mathring{\mathcal U}\subset\mathcal U$ be two compact neighbourhoods of the origin, and let $r>0$ be such that the $n$-fold product of open disks of radius $r$ centered at $y$ is included in $\mathcal U$ for any $y\in\mathcal V$. Let $f=\sum_{\alpha=1}^n f^\alpha\partial_\alpha$ be a vector field on $\mathcal U$ with analytic coefficients, and let us define
\begin{equation*}
	\|f\|_{\mathcal V}:=\mop{sup}_{\alpha=1,\ldots,n}\|f^\alpha\|_{\mathcal V}.
\end{equation*}
Then we have:
\begin{equation*}
	\|f\|_{\mathcal U,\mathcal V}\le \frac{n\|f\|_{\mathcal V}}{r}.
\end{equation*}
\end{proposition}

\begin{proof}
This is an immediate application of Lemma \ref{easy-lemma-one} and the Cauchy estimate for the partial derivation operator $\partial_\alpha$, which is immediately derived from the Cauchy integral formula
\begin{equation*}
	\varphi(y)=\varphi(y_1,\ldots,y_n)=\frac{1}{(2i\pi)}\int_{C_\alpha}\frac{\varphi(y_1,\ldots y_{\alpha-1},
	\eta_\alpha, y_{\alpha+1},\ldots,y_n)}{\eta_\alpha-y_\alpha}\,d\eta_\alpha,
\end{equation*}
valid for any $\varphi\in\mathcal A_{\mathcal U}$ and for any $y\in\mathcal V$, where $C_\alpha$ is the circle of radius $r$ in $\mathbb C$ centered at $y_\alpha$, counterclockwise oriented.
\end{proof}

\begin{corollary}\label{iterate-estimates}
Let $0\in\mathcal V\subset\mathring{\mathcal U}\subset\mathcal U$ be two compact neighbourhoods of the origin, and let $r>0$ be such that the open polydisk of radius $r$ centered at $y$ is included in $\mathcal U$ for any $y\in\mathcal V$. Let $f=\{f_1,\ldots,f_k\}$ be a finite collection of vector fields
$$
	f_j=\sum_{\alpha=1}^n f^\alpha_j\partial_\alpha
$$
on $\mathcal U$ with analytic coefficients, and let us define
\begin{equation*}
	\|f\|_{\mathcal V}:=\mop{sup}_{\alpha=1,\ldots,n \atop j=1,\ldots, k}\|f_j^\alpha\|_{\mathcal V}.
\end{equation*}
and $\|f\|_{\mathcal U}$ similarly. Then we have:
\begin{equation*}
	\|f_1\circ\cdots\circ f_k\|_{\mathcal U,\mathcal V}\le \left(\frac{n\|f\|_{\mathcal U}}{r}\right)^kk^k.
\end{equation*}
\end{corollary}

\begin{proof}
The case $k=1$ is covered by Proposition \ref{estimate-vf}. For $k \ge 2$, we define intermediate compact neighbourhoods
$$
	\mathcal V=\mathcal V_0\subset\mathcal V_1\subset\cdots\subset\mathcal V_k=\mathcal U
$$
as follows: $\mathcal V_j$ is the closure of the union of the polydisks of radius $r/k$ centered at any point of $\mathcal V_{j-1}$, for any $j\in\{1,\ldots, k-1\}$. The result follows then from Proposition \ref{estimate-vf} and the $k$-fold iteration of Lemma \ref{easy-lemma-two} associated with these data, as well as from the obvious inequality $\|f\|_{\mathcal V_j}\le \|f\|_{\mathcal U}$ for any $j=1,\ldots, k$.
\end{proof}


\subsection{Convergence of a formal solution}
\label{ssect:convergence}

We address the question whether the formal diffeomorphism $\mathbf Y_{st}:=\#\mathcal F_f(\mathbb Y_{st})$ converges at least for $|t-s|$ sufficiently small. Any homogeneous space $\mathcal M$ under the action of a finite-dimensional Lie group has a canonical analytic structure. We denote by $C^\omega(\mathcal M,V)$ the space of weakly analytic maps form $\mathcal M$ into a vector space $V$. We suppose that the data $f=\{f_j\}_{j=1}^d$ are analytic maps from $\mathcal M$ to $\mathfrak g$, thus yielding analytic vector fields $\#f_j$ on $\mathcal M$. Choosing $y\in\mathcal M$ and two compact chart neighbourhoods $\mathcal U,\mathcal V$ such that $y\in\mathcal V\subset \mathring{\mathcal U}$, we have to prove that the operator norm $\|\mathbf Y_{st}\|_{\mathcal U,\mathcal V}$ is finite for small $h=t-s$.

\medskip

\noindent Choosing a basis $(E_\alpha)_{\alpha=1,\ldots,N}$ of the Lie algebra $\mathfrak g$, we have:
\begin{equation}\label{fj-ebeta}
	f_j=\sum_{\beta=1}^N\wt f_j^\beta E_\beta,
	\hskip 12mm 
	\#E_\beta=\sum_{\alpha=1}^n \varepsilon^\alpha_\beta\partial_\alpha,
\end{equation}
where the coefficients $\wt f_j^\beta$ and $\varepsilon^\alpha_\beta$ are analytic on $\mathcal U$, and where 
\begin{equation*}
	\|\wt f\|_{\mathcal V}:=\mop{sup}_{j=1,\ldots,d \atop \beta=1,\ldots,N}\|\wt f_j^\beta\|_{\mathcal V}.
\end{equation*}

\begin{theorem}
There exists a positive constant $C_{\mathcal U,\mathcal V}$ such that for any $A$-decorated rooted planar forest $\sigma=\sigma_1\cdots\sigma_k$ with connected components $\sigma_j=B_+^{a_j}(\tau_j)$, the following estimates hold:
\begin{equation}
\label{main-estimate-one}
	\|\wt f_\sigma^{\bm\beta}\|_{\mathcal V}
	\le \tau_1!\cdots\tau_k!C_{\mathcal U,\mathcal V}^{|\sigma|-k}\|\wt f\|_{\mathcal U}^{|\sigma|},
\end{equation}
where the coefficients $\wt f_\sigma^{\bm\beta}\in C^\omega(\mathcal U)$ are considered with respect to the Poincar\'e--Birkhoff--Witt basis:
\begin{equation*}
	\mathcal F_{\sigma}=\sum_{\bm{\beta}\in\{1,\ldots,N\}^k \atop \beta_1\le\cdots
	\le\beta_k}\wt f_{\sigma}^{\bm\beta}E_{\bm\beta},
\end{equation*}
and
\begin{equation}\label{main-estimate-two}
	\|\#\mathcal F_{\sigma} \|_{\mathcal U,\mathcal V}
	\le \sigma!C_{\mathcal U,\mathcal V}^{|\sigma|}\|\wt f\|_{\mathcal U}^{|\sigma|}.
\end{equation}
\end{theorem}

\begin{proof}
Let us first treat the case $|\tau|=1$, i.e., $\tau=\racine_j,\,j=1,\ldots,d$. Estimate \eqref{main-estimate-one} holds by definition of $\|\wt f\|_{\mathcal V}$. Applying Proposition \ref{estimate-vf} we have:
\begin{equation}
\label{estimate-sharp-ebeta}
	\|\#E_\beta\|_{\mathcal U,\mathcal V}\le\frac{n\|\varepsilon\|_{\mathcal V}}{r},
\end{equation}
where $r>0$ is chosen so that any polydisk of radius $r$ centered at a point of $\mathcal V$ is included in $\mathcal U$, and where
\begin{equation*}
	\|\varepsilon\|_{\mathcal V}:=\mop{sup}_{\alpha=1,\ldots,n\atop \beta=1,\ldots,N}\|\varepsilon^\alpha_\beta\|_{\mathcal V}.
\end{equation*}
Applying Estimate \eqref{estimate-sharp-ebeta} and Lemma \ref{easy-lemma-one} we get the estimates:
\begin{equation}
\label{estimates-sharp-fj}
	\|\#f_j\|_{\mathcal U,\mathcal V}\le \frac{nN\|\wt f\|_{\mathcal V}\|\varepsilon\|_{\mathcal V}}{r}.
\end{equation}
We introduce the constant
\begin{equation}\label{constant-cuv}
	C_{\mathcal U,\mathcal V}:=e\frac{nN\|\varepsilon\|_{\mathcal U}}{r},
\end{equation}
so that we immediately get
\begin{equation}\label{estimate-degree-one}
\|\#f_j\|_{\mathcal U,\mathcal V}\le C_{\mathcal U,\mathcal V}\|\wt f\|_{\mathcal V}\le C_{\mathcal U,\mathcal V}\|\wt f\|_{\mathcal U},
\end{equation}
which is estimate \eqref{main-estimate-two}. Let us now proceed by induction to the higher degree case. The necessity of the extra Euler prefactor $e=2,71828...$ in \eqref{constant-cuv} will appear in the proof, as a consequence of the inequality $k^k\le e^kk!$ coming from Stirling's formula. For any decorated planar forest $\sigma=\sigma_1\cdots\sigma_k$ with connected components $\sigma_j$, we can write its decomposition in the Poincar\'e--Birkhoff--Witt basis:
\begin{equation}\label{pbw}
	\mathcal F_{\sigma}=\sum_{\bm{\beta}\in\{1,\ldots,N\}^k\atop\beta_1\le\cdots\le\beta_k}
	\wt f_{\sigma}^{\bm\beta}E_{\bm\beta}
\end{equation}
with $\wt f_{\sigma}^{\bm\beta}=\wt f_{\sigma_1}^{\beta_1}\cdots \wt f_{\sigma_k}^{\beta_k}$ and $E_{\bm\beta}=E_{\beta_1}\cdots E_{\beta_k}\in\mathcal U(\mathfrak g)$. Two cases occur for higher-degree forests:

\begin{enumerate}
\item \textsl{First case:} the decorated forest $\tau$ is not a tree, i.e., $k\ge 2$. In this case we have, using the induction hypothesis on each connected component,
\begin{eqnarray*}
	\|\wt f_{\sigma}^{\bm\beta}\|_{\mathcal V}&\le&\prod_{j=1}^k\|\wt f_{\sigma}^{\beta_j}\|_{\mathcal U}\\
	&\le&\tau_1!\cdots\tau_k!C_{\mathcal U,\mathcal V}^{|\sigma|-k}\|\wt f\|_{\mathcal U}^{|\sigma|}.
\end{eqnarray*}
From decomposition \eqref{pbw} and Proposition \ref{iterate-estimates} we get then:
\begin{eqnarray*}
	\|\#\mathcal F_{\sigma}\|_{\mathcal U,\mathcal V} 
	&\le& \tau_1!\cdots\tau_k!C_{\mathcal U,\mathcal V}^{|\sigma|-k}\|\wt f\|_{\mathcal U}^{|\sigma|} \sum_{\bm{\beta}\in\{1,\ldots,N\}^k\atop\beta_1\le\cdots\le\beta_k}\|E_{\bm\beta}\|_{\mathcal U,\mathcal V}\\
	&\le& N^k \tau_1!\cdots \tau_k!C_{\mathcal U,\mathcal V}^{|\sigma|-k}\|\wt f\|_{\mathcal U}^{|\sigma|}\left(\frac{n\|\varepsilon\|_{\mathcal U}}{r}\right)^kk^k\\
	&\le & k!\tau_1!\cdots\tau_k! C_{\mathcal U,\mathcal V}^{|\sigma|}\|\wt f\|_{\mathcal U}^{|\sigma|}\\
	&\le &\sigma!C_{\mathcal U,\mathcal V}^{|\sigma|}\|\wt f\|_{\mathcal U}^{|\sigma|}.
\end{eqnarray*}
The last inequality derives from the recursive formula \eqref{rec-fac-two} for the planar factorial.

\item \textsl{Second case:} $k=1$, i.e., the decorated forest is a tree $\sigma=B^+_{a_j}(\tau)$. From the definition
\begin{equation}
	\mathcal F_{\sigma}=\mathcal F_{\tau}\rhd f_j
\end{equation}
we get
\begin{equation}
	\wt f_{\sigma}^\beta=(\#\mathcal F_{\tau}).\wt f_j^\beta
\end{equation}
for any $\beta\in\{1,\ldots, N\}$. Applying the induction hypothesis to $\tau$ and Lemma \ref{easy-lemma-one}, we get:
\begin{eqnarray*}
	\|\wt f_{\sigma}^\beta\|_{\mathcal U} 
	&\le&\|\#\mathcal F_{\tau}\|_{\mathcal U,\mathcal V}\|f_j^\beta\|_{\mathcal U}\\
	&\le& C_{\mathcal U,\mathcal V}^{|\tau|}\tau!\|\wt f\|_{\mathcal U}^{|\tau|+1}\\
	&\le& C_{\mathcal U,\mathcal V}^{|\sigma|-1}\tau!\|\wt f\|_{\mathcal U}^{|\sigma|}.
\end{eqnarray*}
Finally, from \eqref{pbw} in the special case $k=1$ we derive:
\begin{eqnarray*}
	\|\#\mathcal F_{\sigma}\|_{\mathcal U,\mathcal V} 
	&\le& \tau!C_{\mathcal U,\mathcal V}^{|\sigma|-1}\|\wt f\|_{\mathcal U}^{|\sigma|} \sum_{\beta\in\{1,\ldots,N\}}\|E_{\beta}\|_{\mathcal U,\mathcal V}\\
	&\le& N\tau!C_{\mathcal U,\mathcal V}^{|\sigma|-1}\|\wt f\|_{\mathcal U}^{|\sigma|}
	\left(\frac{n\|\varepsilon\|_{\mathcal U}}{r}\right)\\
	&\le & \tau! C_{\mathcal U,\mathcal V}^{|\sigma|}\|\wt f\|_{\mathcal U}^{|\sigma|}\\
	&\le &\sigma!C_{\mathcal U,\mathcal V}^{|\sigma|}\|\wt f\|_{\mathcal U}^{|\sigma|}.
\end{eqnarray*}

\end{enumerate}
\end{proof}

\begin{corollary}
The series $\|\mathbf Y_{st}\|_{\mathcal U,\mathcal V}$ is dominated by a series of Gevrey type $1-\gamma$ with respect to the variable $|t-s|^\gamma$.
\end{corollary}
\begin{proof}
Recall that a power series $\sum_{k\ge 0} b_kx^k$ is of Gevrey type $\beta\ge 0$ if and only if there exists a constant $C>0$ such that
\begin{equation}
|b_k|\le C^k(k!)^\beta.
\end{equation}
The series $\mathbf Y_{st}$ is given by $\sum_{k\ge 0}a_k$, with
\begin{equation*}
	a_k=\sum_{\tau\in F^A_{\smop{pl},k}}\langle \wt{\mathbb X}_{st},\,\tau\rangle\#\mathcal F_{\tau}.
\end{equation*}
We compute, using estimates \eqref{estimate-gamma-rough}, \eqref{estimate-qgamma-general}, \eqref{main-estimate-two}, and the majoration of the number of planar $A$-decorated rooted forests of degree $k$ by $(4d)^k$:
\begin{eqnarray*}
	\|a_k\|_{\mathcal U,\mathcal V}
	&=&\left\|\sum_{\tau\in F^A_{\smop{pl},k}}\langle \wt{\mathbb X}_{st},\,\tau\rangle\#\mathcal F_{\tau}\right\|_{\mathcal U,\mathcal V}\\
	&\le&\sum_{ \tau\in F^A_{\smop{pl},k}}\vert\langle \wt{\mathbb X}_{st},\, \tau\rangle\vert.\|\#\mathcal F_{ \tau}\|_{\mathcal U,\mathcal V}\\
	&\le &\sum_{ \tau\in F^A_{\smop{pl},k}}c^{|\tau|}q_\gamma(\tau)\|\#\Cal F_\tau\|_{\Cal U,\Cal V}\\
	&\le& \sum_{ \tau\in F^A_{\smop{pl},k}} c^{|\tau|}C_\gamma^{|\tau|}\frac{|\tau!|^{1-\gamma}}{ \tau !} \tau!|t-s|^{\gamma |\tau|}C_{\Cal U,\Cal V}^{|\tau|}\|\wt f\|_{\mathcal V}^{|\tau|}\\
	&\le& \left(4dcC_\gamma C_{\Cal U,\Cal V}\|\wt f\|_{\mathcal V}\right)^{k}(k!)^{1-\gamma}|t-s|^{\gamma k}.
\end{eqnarray*}
\end{proof}

\ignore{
Let $G_{\mathcal H}$ be the set of characters of the algebra $\mathcal H=\mathcal H^A_{\smop{MKW}}$, and let $\mathfrak g_{\mathcal H}$ be the set of infinitesimal characters of $\mathcal H$. There are two pro-unipotent group laws on $G_{\mathcal H}$, respectively given by the Grossman--Larson product $*$ (dual of the admissible cut coproduct) and the concatenation of forests denoted by a dot $\cdot$ (dual of the deconcatenation coproduct), corresponding to two Lie algebra structures $[\![-,-]\!]$ and $[-,-]$ on $\mathfrak g_{\mathcal H}$. The two following diagrams commute:
$$
	\xymatrix{\mathfrak g_{\mathcal H}\ar[r]^{\exp^*}\ar[d]^{\mathcal F} &G_{\mathcal H}\ar[d]^{\mathcal F}\\
C^{\omega}(\mathcal M,\mathfrak g)[[h]]\ar[r]^{\exp^*} &C^{\omega}(\mathcal M,G)[[h]]}
\hskip 12mm
	\xymatrix{\mathfrak g_{\mathcal H}\ar[r]^{\exp^\cdot}\ar[d]^{\mathcal F} &G_{\mathcal H}\ar[d]^{\mathcal F}\\
C^{\omega}(\mathcal M,\mathfrak g)[[h]]\ar[r]^{\exp^\cdot} &C^{\omega}(\mathcal M,G)[[h]]}
$$
On the bottom lines, $\exp^*$ is the Grossman--Larson exponential reflecting the (formal) flows of vector fields, whereas $\exp^\cdot$ is the pointwise exponential from $\mathfrak g$ to $G$. The righthand side diagram will be more manageable for our purposes: indeed, owing to the fact that the Eulerian idempotent $E=\log^\cdot(\mop{Id})$ is the projection defined by $E(\tau)=\tau$ if $\tau$ is a tree and $E(\tau)=0$ if $\tau=\mathbf 1$ or $\tau$ is a non-connected forest, we have
\begin{eqnarray*}
	\log^\cdot\mathbb Y_{st}
	&=&\mathbb Y_{st}\circ E\\
	&=&\sum_{\tau\in\mathcal T^A_{\smop{pl}}}\langle \wt{\mathbb X}_{st},\,\tau\rangle\tau.
\end{eqnarray*}
}

\begin{corollary}\label{Taylor-cv}
In the case when the driving path $X$ is Lipschitz, i.e., if $\gamma=1$, the norm $\|\mathbf Y_{st}\|_{\mathcal U,\mathcal V}$ is finite for small $h=t-s$.
\end{corollary}



\section*{Appendix: The sewing lemma}
\label{sect:couture}

Let $S,T$ be two real numbers with $S<T$. A map $\Phi$ from $[S,T] \times [S,T]$ into a vector space $B$ is \textsl{additive} if it verifies the chain rule $\Phi(s,t)=\Phi(s,u)+\Phi(u,t)$ for any $s,u,t\in[S,T]$. In that case there obviously exists a map $\varphi:[S,T]\to B$, unique up to an additive constant, such that $\Phi(s,t)=\varphi(t)-\varphi(s)$. Indeed, choose an arbitrary origin $o \in [S,T]$ and set $\varphi(t):=\Phi(o,t)$.

\medskip

Loosely speaking, the sewing lemma stipulates that, under an appropriate completeness assumption on the vector space $B$, a {\sl{nearly}} additive map $(s,t)\mapsto\mu(s,t)$ is {\sl{nearly}} given by a difference $\varphi(t)-\varphi(s)$, in the sense that if $\mu(s,t)-\mu(s,u)-\mu(u,t)$ is small, then there is a unique $\varphi$, defined up to an additive constant, such that $\mu(s,t)-\varphi(t)+\varphi(s)$ is small. In view of the importance of this result, we give an account of it in the precised version given by Gubinelli, together with a detailed proof adapted from \cite{FP2006}. For the original proof, see \cite[Appendix A1]{Gubi2004}.

\begin{proposition}\cite[Proposition 1]{Gubi2004}\label{prop:fdpsewing}
Let $\mu$ be a continuous function on a square $[S,T] \times [S,T]$ with values in a Banach space $B$, and let $\varepsilon>0$. Suppose that there exist a positive integer $n$ and two collections $a_i,b_i\ge 0$ indexed by $i\in\{1,\ldots,n\}$, with $a_i+b_i=1+\varepsilon$, such that $\mu$ satisfies:
\begin{equation}
\label{fdpmu}
	\vert\mu(s,t)-\mu(s,u)-\mu(u,t)\vert\le \sum_{i=1}^nC_i\vert t-u\vert^{a_i}\vert u-s\vert^{b_i}
\end{equation}
for any $s,t,u\in[S,T]$ with $s\le u\le t$ or $t\le u\le s$, where the $C_i$'s are positive constants. Then there exists a function $\varphi \colon [S,T] \to B$, unique up to an additive constant, such that:
\begin{equation}\label{fdp-sewing}
	\vert\varphi(t)-\varphi(s)-\mu(s,t)\vert\le C'\vert t-s\vert^{1+\varepsilon},
\end{equation}
where the best constant in \eqref{fdp-sewing} is
$$	
	C':=\frac{1}{2^{1+\varepsilon}-2}\sum_{i=1}^n C_i.
$$
\end{proposition}

The proof, adapted from reference \cite{FP2006}, is based on dyadic decompositions of intervals. A sequence $(\mu_n)_{n\ge 0}$ of continuous maps from $[S,T] \times [S,T]$ into $B$ is defined by $\mu_0=\mu$ and
\begin{equation}
\label{mun}
	\mu_n(s,t):=\sum_{i=0}^{2^n-1}\mu(t_i,t_{i+1})
\end{equation}
with $t_i=s+i(t-s)2^{-n}$. Denoting by $C$ the sum $C_1+\cdots+ C_n$, the estimate
$$
	\vert\mu_{n+1}(s,t)-\mu_n(s,t)\vert\le C2^{-n\varepsilon -1-\varepsilon}\vert t-s\vert^{1+\varepsilon}
$$
holds, which is easily seen by applying \eqref{fdpmu} to each of the $2^n$ intervals in \eqref{mun}. Hence the sequence $(\mu_n)_{n\ge 0}$ is Cauchy in the complete metric space $\Cal C([S,T]^2,B)$ of continuous maps from $[S,T] \times [S,T]$ into $B$ endowed with the uniform convergence norm:
$$
	\|f\|:=\mop{sup}_{(s,t)\in[S,T]^2}\|f(s,t)\|_B,
$$
and thus converges uniformly to a limit $\Phi$, which moreover verifies:
\begin{equation}
\label{est-phi}
	\vert\Phi(s,t)-\mu(s,t)\vert\le 2^{-1-\varepsilon}C\vert t-s\vert^{1+\varepsilon}
	\sum_{n\ge 0}2^{-n\varepsilon}=C\vert t-s\vert^{1+\varepsilon}\frac{1}{2^{1+\varepsilon}-2}.
\end{equation}

\begin{lemma}
The limit $\Phi$ is additive, that is, it satisfies
$$
	\Phi(s,t)=\Phi(s,u)+\Phi(u,t)
$$
for any $s,u,t\in[S,T]$.
\end{lemma}

\begin{proof}
From $\mu_{n+1}(s,t)=\mu_n\big(s,(s+t)/2\big)+\mu_n\big((s+t)/2,t\big)$ we get that $\Phi$ is \textsl {semi-additive}, i.e., it satisfies
$$
	\Phi(s,t)=\Phi\big(s,(s+t)/2\big)+\Phi\big((s+t)/2,t\big)
$$
for any $s,t\in[S,T]$. Moreover, $\Phi$ is the unique semi-additive map satisfying estimates \eqref{est-phi}. Indeed, if $\Psi$ is another one, then
\begin{eqnarray*}
	\vert (\Phi-\Psi)(s,t)\vert
				&=&\left\vert\sum_{i=0}^{2^n-1}(\Phi-\Psi)(t_{i+1}-t_i)\right\vert\\
				&\le&2C'\sum_{i=0}^{2^n-1}\vert t_{i+1}-t_i\vert^{1+\varepsilon}\\
				&\le& 2C'\vert t-s\vert 2^{-n\varepsilon}
\end{eqnarray*}
with $C'=C/(2^{1+\varepsilon}-2)$. Hence $\Psi=\Phi$ by letting $n$ go to infinity. Now, if $r$ is any positive integer, then the map $\Psi_r$ defined by
$$
	\Psi_r(s,t)=\sum_{j=0}^{r-1}\Phi(t_j,t_{j+1}),
$$
with $t_j=s+j(t-s)/r$, is also semi-additive, hence $\Psi_r=\Phi$. From this we easily get
$$
	\Phi(s,t)=\Phi(s,u)+\Phi(u,t)
$$
for any rational barycenter $u$ of $s$ and $t$, i.e., such that $u=as+(1-a)t$ with $a\in[0,1]\cap\mathbb Q$. Additivity of $\Phi$ follows from continuity.
\end{proof}

The proof of Proposition \ref{prop:fdpsewing} follows by choosing $\varphi(t):=\Phi(o,t)$ for any fixed but arbitrary $o\in[S,T]$. Uniqueness of $\varphi$ up to an additive constant follows immediately from the uniqueness of the additive function $\Phi$ satisfying estimate \eqref{est-phi}.



\end{document}